\documentclass[reqno,a4paper]{article}
\usepackage{mathrsfs}
\usepackage{amssymb,amsmath,amsfonts,latexsym}
\usepackage{amsthm}
\usepackage{graphicx,float,epsfig,color,fancyhdr}
\usepackage{framed}
\usepackage{todonotes}
\usepackage{comment}
\usepackage{multirow}
\usepackage{relsize}
\usepackage[utf8x]{inputenc}
\definecolor{lightblue}{rgb}{0.22,0.45,0.70}
\usepackage[colorlinks=true,breaklinks=true,linkcolor=lightblue,citecolor=lightblue]{hyperref}
\usepackage{float}
\usepackage{comment}
\newtheorem{remark}{Remark}[section]
\newtheorem{lemma}{Lemma}[section]
\newtheorem{theorem}{Theorem}[section]
\newtheorem{proposition}{Proposition}[section]

\newtheorem{problem}{Problem}

\usepackage{soul}

\newcommand{\de}{~\textrm{d}e}
\newcommand{\df}{~\textrm{d}f}
\newcommand{\dP}{~\textrm{d}P}

\def\E{P}

\begin{document}
\title{A virtual element method on polyhedral meshes for the sixth-order elliptic problem}
\author{
Franco Dassi\thanks{Dipartimento di Matematica e Applicazioni, Università degli studi di Milano Bicocca, Via Roberto Cozzi 55 - 20125 Milano, Italy. E-mail:
{\tt franco.dassi@unimib.it}},\quad
David Mora \thanks{ Departamento  de  Matem\'aticas, Universidad del B\'io-B\'io, Concepci\'on, Chile. E-mail:
{\tt dmora@ubiobio.cl}.}
Carlos Reales\thanks{Departamento  de  Matem\'aticas, Universidad de C\'ordoba, Monter\'ia,  Colombia. E-mail:
{\tt creales@correo.unicordoba.edu.co.}},\quad
Iv\'an Vel\'asquez \thanks{ Departamento  de  Ciencias  B\'asicas, Universidad del Sin\'u-El\'ias   Bechara   Zain\'um,    Monter\'ia,  Colombia. E-mail:
{\tt ivanvelasquez@unisinu.edu.co}.}
}
\date{}
\maketitle
%-----------------------------------------------
\begin{abstract}
In this work we analyze a virtual element method on polyhedral meshes for solving the sixth-order elliptic problem with simply supported boundary conditions. 
We apply the Ciarlet-Raviart  arguments to introduce an auxiliary unknown $\sigma:=-\Delta^2 u$ and to search the main uknown $u$ in the $H^2\cap H_0^1$ Sobolev space. 
The virtual element discretization is well possed on a $C^1\times C^0$ virtual  element spaces. We also provide the convergence and error estimates results. Finally, we report a series of numerical tests to verify the performance of numerical scheme.
\end{abstract}
%-----------------------------------------------
\maketitle
%-----------------------------------------------
%-----------------------------------------------------------------------
\setcounter{equation}{0}
\section{Introduction}
\label{SEC:Introduction}
This paper deals the sixth-order elliptic problem with simply supported boundary conditions, which it reads as follows:
\begin{subequations}\label{SixthEq}
\begin{align}
-\Delta^3 u = f &\qquad \mbox{in}\,\, \Omega, \label{SixthEq_a}\\
u=\Delta u=0&\qquad \mbox{on}\,\, \partial\Omega, \label{SixthEq_b}\\
\Delta^2 u=0&\qquad \mbox{on}\,\, \partial\Omega, \label{SixthEq_c}
\end{align}
\end{subequations}
where $\Omega\subset \mathbb{R}^{3}$ is a bounded domain with polyhedral boundary $\Gamma:=\partial \Omega$, and $f\in H^{-1}(\Omega)$. 
There are several numerical methods for solving the sixth-order elliptic equation \eqref{SixthEq_a} \cite{DroniouSixht2019,GudiSixthIPG2011,LiuXu,Chang2005}.  
For instance, the authors of \cite{DroniouSixht2019} studied a mixed finite element formulation based on the Ciarleth-Raviart  method, they decrease the regularity of the main unknown to approximate it with $H^1-$conforming  Lagrange spaces. 
In \cite{GudiSixthIPG2011} it was proposed a $C^0-$interior penalty method to studied the sixht-order elliptic equation with clamped boundary conditions.

Sixth-order problems has several application in different areas such as the thin-film equations (\cite{Barret,Liu}), the phase field crystal model \cite{Backofen,Cheng},fluid flows \cite{Tagliabue}
geometric design \cite{KuUgWi},\cite{YiZhYou}. among others. 
In the course of this work it will be possible to show the mathematical differences that arise when considering three-dimensional domains. Besides the natural interest generated by the formulations in this type of domain, interesting applications of this type of problem, as geometric design previously mentioned , can be highlighted. The use of elliptical PDEs for shape design is very different from conventional spline-based method. The idea behind this method is that the design of the shape is effectively treated as a mathematical boundary value problem, i.e. the shapes are produced by finding the solutions to a suitably chosen elliptic PDE that satisfies certain boundary conditions \cite{Ugail} . This technique has achieved great relevance in the generation of surfaces given the discretization of the operator associated with the elliptic PDE, it is a process of averaging the solution neighborhood of the PDE that guarantees that the surface obtained will have a certain degree of smoothness depending on the order of the PDE \cite{ Castro}. In \cite{Bloor1989} introduces the original formulation of the Bloor Wilson PDE method, which consists of producing a parametric surface $X(u, v)$ by finding the solution to a PDE of the form 
\begin{equation} 
\left(\dfrac{\partial^2}{\partial u^2}+\dfrac{\partial^2}{\partial v^2}\right)^{r} X(u, v) = 0,
\end{equation}
where $u$ and $v$ represent the parametric coordinates of the surface ates, which are then mapped into physical space; namely., $(x(u, v), y(u, v), z(u, v)).$ $r$ determines the order of the PDE and the smoothness of the surface.

The virtual element method (VEM) is a new powerful technology in the numerical analysis for solving partial differential equations introduced for first time in \cite{BBCMMR2013}. The VEM can be see an extension of the standard finite element method on polygonal and polyhedral meshes. The VEM on polyhedral meshes  has aroused a very significant interest in the scientific community dedicated to solving problems in different areas: fluid \cite{Dassi2020stokes3D,DASSI_pressure3D,DASSI_Bend3DDarcy}, elasticity, \cite{DASSI_LovadinaReissner}, electromagnetic \cite{Dassi3dLowestorderMagnetostatic,DASSI_3dMaxwell,Dassi3DMagnetostaticsSIAM}, eigenvalue problems~\cite{DV_camwa2022,GMV2018,MVsiam2021}.

In this work we propose and analyze a new mixed variational formulation on $H^2\times H^1$ sobolev space in $3\mathrm{D}$ for solving the  sixth-order elliptic equation \eqref{SixthEq_a} with simply supported boundary conditions (cf.  \eqref{SixthEq_b}-\eqref{SixthEq_c}). 
We follow the arguments presented in \cite{ciarlet} to introduce  an auxiliary unknown $\sigma=-\Delta^2 u$ in $ \Omega$ to reduce the sixth-order equations \eqref{SixthEq_a}-\eqref{SixthEq_c} to  fourth order problem. 
Then, we apply the $C^1$ and $C^0-$ VEMs established on \cite{HO_C0VEM_Polyhedral} and \cite{C1VEM_Polyhedral}, respectively, to propose a well posed discrete virtual element scheme. We derive optimal error estimates in $H^2$ and $H^1-$norms. In addition, we apply the duality arguments to obtain additional regularity results both unknown $u$ and $\sigma$. 
Since conforming elements for solving the sixth-order elliptic equation \eqref{SixthEq_a}-\eqref{SixthEq_c} need $C^2$ elements which have a very high computational cost to be implemented in   Finite Element Method (FEM) is that our discrete scheme results an interesting alternative.

The outline of this is paper is as follows. 
In Section~\ref{Sec:SixthCont}, we introduce the weak formulation associated to the sixth-order elliptic problem and prove wellposedness using a decomposition into two varational problems. Next, in Section~\ref{VEMspaces}, we present the definitions of the $C^1$ and $C^0$ virtual element spaces.
Then in Section~\ref{SEC:DisSpecProb}, a virtual element discrete formulation by mimicking the analysis developed for the continuous problem are presented in this section. 
In addition, in Section~\ref{SecConvErrEst} we prove that the numerical scheme provides a correct approximation and establish optimal order error estimates for approximated solutions. 
Finally, in Section~\ref{SEC:NumericalResults}, 
we report some numerical tests that confirm the theoretical analysis developed.

\paragraph{Notations} Throughout of this work we will use standard notations for Sobolev spaces, norms and semi-norms~\cite{adams2003sobolev}. 
Moreover, we will use the notation $a \lesssim b$ to mean that there exists a positive constant $C$ independent of $a, b$, and the size of the elements in the mesh, such that $a\leq C b$. 
We stress that such constant can be take different values in the different occurrences of ``$\lesssim$''.
We use the standard notations for operators gradient, $\nabla$, normal derivative, $\partial_{\boldsymbol{n}}$, and the hessian, $\nabla^2$. 
In addition, given a positive integer $d=1,2$ or 3, for all $\mathcal{O}\subseteq \mathbb{R}^d$,
we denote the polynomial space in $d$-variables of degree lower or equal to $k$ as $\mathbb{P}_{k}(\mathcal{O})$.

\section{The sixth-order elliptic continuous problem}\label{Sec:SixthCont}
We inspire on Ciarlet and Raviarth method (see e.g. \cite[Section~7.1]{ciarlet}) to write a mixed weak formulation associated to the system of equations \eqref{SixthEq_a}-\eqref{SixthEq_b}, we introduce the auxiliary unknown
\begin{equation}
    \sigma:=-\Delta^2 u \in \Omega.\label{def_sigma}
\end{equation}
It is well known that since $\Omega$ is a convex domain we have that $u\in H^{5}\cap H_0^1(\Omega)$ if $f\in H^{-1}(\Omega)$, and it holds that
\begin{equation}
    ||u ||_{5,\Omega}\lesssim ||f||_{-1,\Omega}\label{RegOfu}
\end{equation}
(see for instance \cite{AMV20,gazzola}). Then, since $u\in H^5(\Omega)\cap H_0^1(\Omega)$ (cf. \eqref{RegOfu}) it holds $\sigma\in H^1(\Omega)$. 

Next, we apply the Laplacian operator on the identity $\eqref{def_sigma}$, and then, we apply the equations \eqref{SixthEq_a} and \eqref{SixthEq_c} to arrive to the following second-order problem for $\sigma$
\begin{subequations}\label{SecondEq}
\begin{align}
\Delta \sigma=f\quad \mbox{in }\Omega,\label{SecondEq_a}\\
\sigma=0\quad \mbox{on }\Gamma.\label{SecondEq_b}
\end{align}
\end{subequations}
Then, we multiply the equation \eqref{SecondEq_a} by $ v   \in H_0^1(\Omega)$, integrate by parts, and use the boundary condition  \eqref{SecondEq_b} to  arrive to the following weak formulation. Find $\sigma \in H_0^1(\Omega)$ such that 
\begin{equation}
\int_\Omega \nabla \sigma \cdot \nabla  v  \ d\Omega=- \big\langle f,v\big\rangle_{-1,1,\Omega}  \quad \forall  v   \in H_0^1(\Omega).\label{VarForH1Cont}     
\end{equation}
Whenever the data $f\in L^2(\Omega)$, we can rewrite the duality pairing between $L^2(\Omega)$ and itself from the following $L^2-$inner product  
%\begin{equation*}
$\langle  f,v \rangle =\int_{\Omega}fv.$
%\end{equation*}

On the other hand, the equations \eqref{SixthEq_a},\eqref{SixthEq_b} and \eqref{def_sigma} yields the following fourth-order problem for $u$ with simply supported boundary conditions.
\begin{subequations}\label{FourtEq}
\begin{align}
\Delta^2 u
+
\sigma
=0 \quad \mbox{in } \Omega,\label{FourtEq_a}\\
u=\Delta u=0\quad \mbox{on } \Gamma\label{FourtEq_b}.
\end{align}
\end{subequations}

It is well known that system of equations \eqref{FourtEq_a}-\eqref{FourtEq_b} gives the following weak formulation (see for instance \cite{ciarlet,G}). Find $u\in H^2(\Omega)\cap H_0^1(\Omega)$ such that
\begin{equation}
\int_\Omega \nabla^2u :\nabla^2\psi  \ d\Omega
+
\int_\Omega \sigma \psi \ d\Omega
=0\quad \forall \psi\in H^2(\Omega)\cap H_0^1(\Omega).\label{VarForH2Cont}
\end{equation}

Now, we will write a weak  formulation associated to the problems \eqref{SecondEq} and \eqref{FourtEq}. First, we define the spaces $V:=H_0^1(\Omega)$ and $\Psi:=H^2(\Omega)\cap H_0^1(\Omega)$ endowed with their standard norms. 

Next, from \eqref{VarForH1Cont} and \eqref{VarForH2Cont} a variational formulation associated to the problems \eqref{SecondEq} and \eqref{FourtEq} (and as consequence a variational formulation for \eqref{SixthEq}) can be reads as follows. Find $\sigma\in V$ and $u \in  \Psi$ such that 
\begin{equation}\label{VarForContH2H1}
a^{\Delta}(u,\psi) +a^{\nabla}(\sigma, v  )+a^{0}(\sigma,\psi)=F( v  )\quad  \forall  v   \in V, \ \forall \psi\in \Psi,
\end{equation}
where we have defined the bilinear and linear forms
\begin{align}
&a^{\Delta}: \Psi\times \Psi\to \mathbb{R};
\qquad a^{\Delta}(u,\psi):=\int_{\Omega}\nabla^2 u:\nabla^2 \psi \  d\Omega
&\quad \forall u,v\in \Psi,
\label{defaD}\\
&a^{\nabla}:V\times V\to \mathbb{R};
\qquad a^{\nabla}(\sigma, v  ):=\int_{\Omega}\nabla \sigma\cdot \nabla v  \  d\Omega
&\quad \forall \sigma, v   \in V,
\label{defan}\\
&a^{0}:V\times \Psi\to \mathbb{R};
\qquad a^{0}(\sigma,\psi):=\int_{\Omega}\sigma \psi\  d\Omega
&\quad \forall \sigma\in V \ , \forall \psi\in \Psi, 
\label{defa0}\\
&F: V\to \mathbb{R};
\qquad  \qquad  F( \psi  ):=- \big\langle f,v\big\rangle_{-1,1,\Omega} 
&\quad \forall  v   \in V.
\label{defF}
\end{align}

The following result shows some  properties of the bilinear forms defined above.
\begin{lemma}\label{lemboundcont}
The bilinear forms $a^{\Delta}(\cdot,\cdot), a^{\nabla}(\cdot,\cdot),a^{0}(\cdot,\cdot)$, and the linear form $F(\cdot)$ satisfy the following properties.
\begin{subequations}
\begin{align}
&|a^{\Delta}(u,\psi)|\lesssim  ||u ||_{2,\Omega} ||\psi ||_{2,\Omega},\label{Bound_aDCont}\\
&|a^{\nabla}(\sigma, v  )|\lesssim  ||\sigma ||_{1,\Omega} || v   ||_{1,\Omega},\label{Bound_anCont}\\
&|a^{0}(\sigma,\psi)|\lesssim ||\sigma ||_{1,\Omega} ||\psi ||_{2,\Omega},\label{Bound_a0Cont}\\
&|F( v  )|\lesssim ||f ||_{-1,\Omega} || v   ||_{1,\Omega},\label{BoundFCont}\\
&  a^{\Delta}(\psi,\psi)\geq C_{P_2}||\psi||_{2,\Omega}^{2}, \label{ellip_aDCont}\\
& a^{\nabla}(\sigma,\sigma)\geq C_{P_1} || \sigma||_{1,\Omega}^{2},\label{ellip_anCont}
\end{align}
\end{subequations}
for all $\sigma, v  \in V$ and $u,\psi \in \Psi$.
\end{lemma}
\begin{proof}
Estimates \eqref{Bound_aDCont}-\eqref{BoundFCont} follow from  Cauchy-Schwarz inequality. 
Ellipticity estimates \eqref{ellip_aDCont} and \eqref{ellip_anCont} can be obtained from generalized Poincaré inequality applied on $\Psi$ and $V$ spaces, respectively (see for instance  \cite[Theorem 2.31]{gazzola} and \cite{Brezis_A_Func}).
\end{proof}

With the aim to prove the well-posedness of weak formulation~\eqref{VarForContH2H1}, we
decomposed it into the following problems:
\begin{problem}\label{ContProbH1}
Find $\sigma\in V$ such that 
\begin{align}
a^{\nabla}(\sigma, v  )=F( v  ) \quad  \forall  v   \in V.\nonumber
\end{align}
\end{problem}
\begin{problem}\label{ContProbH2}
Find $u\in  \Psi$ such that 
\begin{align}
\displaystyle a^{\Delta}(u,\psi) =-a^{0}(\sigma,\psi) \quad  \forall \psi\in \Psi.\nonumber
\end{align}
\end{problem}

It is clear that the Lax-Milgram theorem and estimates \eqref{Bound_anCont}, \eqref{BoundFCont}, and \eqref{ellip_anCont} allow us to deduce  the well-posedness of Problem~\ref{ContProbH1}. Next, we apply once more the Lax-Milgram theorem, and use the  inequalities \eqref{Bound_aDCont}, \eqref{Bound_a0Cont}, \eqref{ellip_aDCont}, and the uniqueness of $\sigma\in V$ on Problem~\ref{ContProbH1} to obtain the unique solution $u\in \Psi$ of Problem~\ref{ContProbH2}. 
Therefore, we have proved the following theorem.
\begin{theorem}\label{UniqCont}
Given $f\in L^2(\Omega)$ there exists a unique $(\sigma,u)\in V\times \Psi$ such that the weak formulation \eqref{VarForContH2H1} holds true.
\end{theorem}

\subsection{Regularity result}
The following result establishes an additional regularity for $u$.

\begin{lemma}\label{Lema_RegAd}
For all $f\in \L^2(\Omega)$ the unique solution $(\sigma, u)\in V\times \Psi$ of weak formulation~\eqref{VarForContH2H1} satisfies
\begin{equation}
\Vert \sigma\Vert_{1,\Omega}
+
\Vert u\Vert_{5,\Omega}
\leq C \Vert f \Vert_{0,\Omega}.\label{AddRegu_u}
\end{equation}
\end{lemma}

\begin{remark}\label{remark_regul_ad}
We point out that the solution  $\sigma \in V$ does not have additional regularity. This can be represent a problem to obtain  interpolation results for $\sigma$,  however, we will apply the standard duality arguments to obtain overcome this fact.
\end{remark}

\section{Virtual element spaces on polyhedral meshes}\label{VEMspaces}

In this section, we introduce the virtual element spaces on polyhedral meshes to propose a discrete variational formulation associated to the Problems~\ref{ContProbH1} and~\ref{ContProbH2}, and consequently for variational formulation \eqref{VarForContH2H1}. 
Since the weak formulation is defined on $V\times \Psi$ we will consider the $C^0$ and $C^1$ virtual element spaces introduced on \cite{HO_C0VEM_Polyhedral} and \cite{C1VEM_Polyhedral}, respectively.

\paragraph{Mesh assumptions} Let $\Omega_h$ be a discretization of $\Omega$ composed by generic polyhedrons $P$. The convergence theory behind the usage of VEM spaces requires the following mesh assumptions:
\begin{itemize}
\item[{\bf A}$_1$:] each element $P$ is star shaped	with respect to a ball $B_P$ 
whose radius is uniformly comparable with the polyhedron diameter, $h_P$;
\item[{\bf A}$_2$:] each face $f$ is star shaped with respect to a disc $B_f$ 
whose radius is uniformly comparable with the face diameter, $h_f$;
\item[{\bf A}$_3$:] given a polyhedron $P$ all edge lengths and face diameters are uniformly comparable with respect to its diameter $h_P$.
\end{itemize}

We recall the virtual element spaces in three dimension. 
With this end, we will proceed as a standard virtual element framework in three dimension.
For that, we define the virtual spaces and projections over polyhedron's faces.
Then, we move to the bulk spaces and finally we glue all these spaces together to have the global one.
Now, because of details of these constructions are already described on \cite{HO_C0VEM_Polyhedral} and \cite{C1VEM_Polyhedral}, we we will only present a brief description about that.

\subsection{$C^0-$VEM spaces on polyhedral meshes}

\paragraph{$C^0-$Face space}
Given a face $f$ of a generic polyhedron of the mesh $\Omega_h$, we consider the so-called enhanced space on face
\begin{align*}
	V_h(f):=\Big\{
	 v_h   \in H^1(f)\cap C^{0}(f)&: \Delta_{\nu}  v_h    \in \mathbb{P}_1(f), \phantom{m}  v_h   |_{\partial{f}}\in C^0(\partial f),\  v_h   |_{e}\in \mathbb{P}_1(e) \ \forall e\in \partial f\\[0.2em]
	&\phantom{m} \int_f \Upsilon_{f}^{\nabla} v_h   \df=\int_f  v_h   \df \quad \forall q\in \mathbb{P}_{1}(f)\Big\}\,,
\end{align*}
where $\Delta_{\nu}$ is the Laplace operators in the local face coordinates and $\Upsilon_{f}^{\nabla}$ is the projection operator from the virtual spaces to the space of polynomials based on the $H^1$ scalar product and it is defined as $\Upsilon_{f}^{\nabla}:V_h(f)\to\mathbb{P}_{1}(f)$ such that
\begin{equation}
\left\{\begin{array}{rl}
    \mathlarger{\int}_f \nabla (\Upsilon_{f}^{\nabla}  v_h    - v_h   )\cdot \nabla p_1\df&=0\qquad \forall p_1\in \mathbb{P}_1(f),\nonumber\\[0.5em]
    \mathlarger{\int}_{\partial f} (\Upsilon_{f}^{\nabla}  v_h    - v_h   )\de&=0 \nonumber.
        \end{array}\right.
\end{equation}
The degrees of freedom of such spaces are 
the evaluation of the virtual function on all vertexes.
Moreover, such degrees of freedom are sufficient to compute the projection operator $\Upsilon_{f}^{\nabla}$.
We refer the reader to~\cite{HO_C0VEM_Polyhedral} for further details.

%In the next paragraph we will see that both $V_h^{\nabla}(f)$ and $V_h^{\Delta}(f)$ play an important role in the definition of the virtual element space inside the polyhedron $P$. In particular, the key ingredients is the \emph{enhancing} properties:
%$$
%\int_f \Upsilon_{f}^{\nabla}\psi_h\df=\int_f \psi_h\df\qquad\text{and}\qquad
%\int_f \Upsilon_{f}^{\Delta}\psi_h\,p_1\df=\int_f \psi_h\,p_1 \df \quad \forall p_1\in \mathbb{P}_1(f)\,.
%$$

\paragraph{$C^0-$Polyhedron space}

\begin{align*}
	V_h(P):=\Big\{
	 v_h   \in H^1(P) &: \Delta   v_h    \in \mathbb{P}_1(P), \phantom{m}  v_h   |_{\partial{P}}\in C^0(\partial P),\  v_h   |_{f}\in \mathbb{P}_1(f) \ \forall f\in \partial P\\[0.2em]
	&\phantom{m} \int_P \Upsilon_{P}^{\nabla} v_h   \dP=\int_P  v_h   \dP \quad \forall q\in \mathbb{P}_{1}(P)\Big\}\,.
\end{align*}

A virtual function in $ v_h    \in V_h(\E)$ and the projection operator $\Upsilon_{\E}^{\Delta}$ 
are uniquely determined by the following degrees of freedom:
\begin{itemize}
	\item[${\bf D_0}$:] evaluation of $ v_h   $ at the vertexes of $\E$
\end{itemize} 

The projection operator $\Upsilon_{\E}^{\nabla}$ is essential to define the VEM discrete counterpart of the bi-linear operator $a^\nabla(\cdot,\cdot)$,
see Equation~\eqref{defan}.
However, to set Problem~\ref{ContProbH2}, 
we have to define the projection operator  $\Upsilon_\E^0:V_h(\E)\to\mathbb{P}_1(E)$ for $a^0(\cdot,\cdot)$ as
\begin{equation}
\int_\E \Upsilon_{\E}^0  v_h   \,p_1\dP=\int_\E   v_h   \,p_1\dP\quad \forall p_1\in \mathbb{P}_1(\E)\,.
\nonumber 
\end{equation}

\begin{remark}
From the definition of $V_h(P)$ we have that the for all $ v_h   \in V_h(P)$ the following identity  $\Upsilon_{\E}^\nabla  v_h   =\Upsilon_{\E}^0  v_h   $ holds true (see \cite{HO_C0VEM_Polyhedral} for further details).
\end{remark}

\paragraph{$C^0-$Global space}
The global space is obtained by gluing the local spaces $V_h(\E)$ together, i.e.,
we define
\begin{equation}
V_h:=\Big\{ v_h   \in H^1(\Omega):  v_h   |_{\E}\in V_h({\E})\quad \forall P\in \Omega_h\Big\}\,.
\label{eqn:fullC0Space}
\end{equation}

\subsection{$C^1-$VEM spaces on polyhedral meshes}

\paragraph{$C^1-$Face spaces}
Given a face $f$ of a generic polyhedron of the mesh $\Omega_h$,
we consider the so-called enhanced spaces on faces 
\begin{align*}
	\Psi_h^{\nabla}(f):=\Big\{
	\psi_h\in H^1(f)&: \Delta_{\nu} \psi_h \in \mathbb{P}_0(f), \phantom{m} \psi_h|_{\partial{f}}\in C^0(\partial f),\ \psi_h|_{e}\in \mathbb{P}_1(e) \ \forall e\in \partial f\\[0.2em]
	&\phantom{m} \int_f \Pi_{f}^{\nabla}\psi_h\df=\int_f \psi_h\df \Big\}\,,
\end{align*}
and
\begin{align*}
	\Psi_h^{\Delta}(f):=\Big\{
	\psi_h\in H^2(f) &: \Delta^2_{\tau}\psi_h\in\mathbb{P}_{1}(f), \phantom{m} \psi_h|_{\partial f}\in C^0(\partial f),\ \psi_h|_e\in\mathbb{P}_3(e)\,\,\forall e\in\partial f,\\[0.2em]
	&\phantom{m} \nabla_{\nu} \psi_h|_{\partial f}\in [C^0(\partial f)]^2,\ \partial_{{\boldsymbol{n}}_f^e} \psi_h|_e\in\mathbb{P}_1(e)\,\,\forall e\in\partial f, \\[0.2em]
	&\phantom{m}\int_f \Pi_{f}^{\Delta}\psi_h\,p_1\df=\int_f \psi_h\,p_1 \df, \quad \forall p_1\in \mathbb{P}_1(f) \Big\}\,,
\end{align*}
where $\Delta^2_{\nu}$ is the bi-harmonic operator in the local face coordinates and $\Pi_{f}^{\nabla}$ is the projection operator from the virtual spaces to the space of polynomials based on the $H^2$ scalar product and it is defined as follows:
$\Pi_{f}^{\Delta}:\Psi_h^{\Delta}(f)\to\mathbb{P}_{2}(f)$ such that
\begin{equation}
\left\{\begin{array}{rl}
    \mathlarger{\int}_f \nabla^2 (\Pi_{f}^{\Delta} \psi_h -\psi_h)\cdot \nabla^2 p_2\df&=0\qquad \forall p_2\in \mathbb{P}_2(f),\nonumber\\[0.5em]
    \mathlarger{\int}_{\partial f} (\Pi_{f}^{\Delta} \psi_h -\psi_h)p_1\  \de&=0\qquad \forall p_1\in \mathbb{P}_1(f)\nonumber.
        \end{array}\right.\,
\end{equation}

The degrees of freedom of such spaces are 
the evaluation of the virtual function and its gradient on all vertexes. 
Moreover, such degrees of freedom are sufficient to compute both projection operators $\Pi_{f}^{\nabla}$ and $\Pi_{f}^{\Delta}$.
We refer the reader to~\cite{BDR2019C1Polyhedral} for a more rigorous proof about this fact.

In the next paragraph we will see that both $\Psi_h^{\nabla}(f)$ and $\Psi_h^{\Delta}(f)$ play an important role in the definition of the virtual element space inside the polyhedron $P$. In particular, the key ingredients is the \emph{enhancing} properties:
$$
\int_f \Pi_{f}^{\nabla}\psi_h\df=\int_f \psi_h\df\qquad\text{and}\qquad
\int_f \Pi_{f}^{\Delta}\psi_h\,p_1\df=\int_f \psi_h\,p_1 \df \quad \forall p_1\in \mathbb{P}_1(f)\,.
$$
Indeed such relations are crucial to define all bulk projection operators inside $P$ and, consequently the discrete counterpart of the operators in Equation~\eqref{defaD},~\eqref{defan} and~\eqref{defa0}~\cite{BDR2019C1Polyhedral}.

\paragraph{$C^1-$Polyhedron space}
The virtual element space inside the polyhedron is defined as 
\begin{align*}
\Psi_h(\E):=\Bigg\{\psi_h\in H^2(\E) &: \Delta^2\psi_h\in\mathbb{P}_{2}(\E),\phantom{m} \psi_h|_{S_\E}\in C^0(S_\E), \nabla \psi_h|_{S_\E}\in [C^0(S_\E)]^3,\\[0.5em]
                &\phantom{m} \psi_h|_f\in \Psi_h^{\Delta}(f),\quad\partial_{{\boldsymbol{n}}_P^f} \psi_h|_f\in  \Psi_h^{\nabla}(f),\  \forall f\in\partial\E\\[-0.25em]
                &\phantom{m}\int_\E \Pi_{\E}^{\Delta}\psi_h\,p_2\dP=\int_\E  \psi_h\,p_2\dP\quad \forall p_2\in \mathbb{P}_2(\E)\Bigg\},
\end{align*}
where we exploit the face spaces introduced in the previous section 
to define the virtual function over the face, $\psi_h|_f\in \Psi_h^{\Delta}(f)$,
and the normal component of its gradient, $\partial_{{\boldsymbol{n}}_P^f} \psi_h|_f\in \Psi_h^{\nabla}(f)$.
To define such space, we add the so-called enhancing property,
\begin{equation}
\int_\E \Pi_{\E}^{\Delta}\psi_h\,p_2\dP=\int_\E  \psi_h\,p_2\dP\quad \forall p_2\in \mathbb{P}_2(\E)\,,
\label{eqn:henP}
\end{equation}
that is based on the $H^2$ projection operator 
$\Pi_{\E}^{\Delta}:\Psi_h(\E)\to\mathbb{P}_2(\E)$ 
defined as 
\begin{equation}
\left\{\begin{array}{rl}
\mathlarger{\int}_{\E} \nabla ^2(\Pi_{\E}^{\Delta}\psi_h-\psi_h):\nabla ^2q\dP=0&\quad \forall q\in \mathbb{P}_2(\E)\\[1ex]
\mathlarger{\int}_{\partial \E}(\Pi_{\E}^{\Delta}\psi_h-\psi_h)\,q\df =0&\quad \forall q \in  \mathbb{P}_1(\E)
\end{array}\right.\,.
\label{eqn:deltaProjDef}
\end{equation}
A virtual function in $\Psi_h(\E)$ and the projection operator $\Pi_{\E}^{\Delta}$ 
are uniquely determined by the following degrees of freedom:
\begin{itemize}
	\item[${\bf D_1}$:] evaluation of $\psi_h$ at the vertexes of $\E$,
	\item[${\bf D_2}$:] evaluation of $\nabla \psi_h$ at the vertexes of $\E$.
\end{itemize} 
The projection operator $\Pi_{\E}^{\Delta}$ is essential
to define the VEM discrete counterpart of the bi-linear operator $a^\Delta(\cdot,\cdot)$,
see Equation~\eqref{defaD}.
However, to set Problem~\ref{ContProbH2}, 
we have to define the projection operator  $\Pi_\E^0:\Psi_h(\E)\to\mathbb{P}_2(E)$ for $a^0(\cdot,\cdot)$ as
\begin{equation}
\int_\E \Pi_{\E}^0 \psi_h\,p_2\dP=\int_\E  \psi_h\,p_2\dP\quad \forall p_2\in \mathbb{P}_2(\E)\,.
\label{eqn:L2PDef}
\end{equation}

\begin{remark}
Comparing Equation~\eqref{eqn:L2PDef} and the enhancing property, Equation~\eqref{eqn:henP},
we deduce that in $\Psi_h(\E)$ these two projections coincide, i.e., $\Pi_{\E}^\Delta \psi_h=\Pi_{\E}^0 \psi_h$~\cite{BDR2019C1Polyhedral}.
\label{rem:coinPiDelatPiZero}
\end{remark}

\paragraph{$C^1-$Global space}
Starting from such local spaces, we are ready to define the global one.
Given a decomposition of the computational domain $\Omega$ into polyhedrons, $\Omega_h$,
The global space is obtained by gluing the local spaces $\Psi_h(\E)$ togheter, i.e.,
we define
\begin{equation}
\Psi_h:=\Big\{\psi_h\in H^2(\Omega): \psi_h|_{\E}\in \Psi_h({\E})\Big\}\,.
\label{eqn:fullC1Space}
\end{equation}

\section{Discrete problem}
\label{SEC:DisSpecProb}
In this section we will introduce a virtual element discretization for solving  the variational formulation~\ref{VarForContH2H1}. With this end, we first introduce the following preliminary definitions. 
Given a polyhedral mesh $\Omega_h$, we split the bilinear forms defined on \eqref{defaD}-\eqref{defa0} over each polyhedron $P$. 
\begin{align*}
	a^{\Delta}(u,\psi)=\sum_{\E\in\Omega_h}a_{\E}^{\Delta}(u,\psi)\,,\quad 
	a^{\nabla}(\sigma, v)=\sum_{\E\in\Omega_h}a_{\E}^{\nabla}(\sigma, v)\,\quad\text{and}\quad a^{0}(\sigma,\psi)=\sum_{\E\in\Omega_h}a_{\E}^{0}(\sigma,\psi)\,,
\end{align*}
where 
\begin{equation*}
a_{\E}^{\Delta}(u,\psi):=\int_{P} \nabla^2u :\nabla^2\psi\dP\,,\quad
a_{\E}^{\nabla}(\sigma, v):=\int_{P} \nabla \sigma \cdot \nabla  v\dP\quad\text{and}\quad
a_{\E}^{0}(\sigma,\psi):=\int_{P} \sigma\,\psi\dP\,.
\end{equation*}
To discretize the problems at hand, 
we construct the following bi-linear forms 
\begin{align*}
a_{\E,h}^{\Delta}:\Psi_h(P)\times \Psi_h(P)\to \mathbb{R}\,,\qquad a_{\E,h}^{\Delta}(u_h,\psi_h) 
&:=a_{\E}^{\Delta}(\Pi_{\E}^{\Delta} u_h,\Pi_{\E}^{\Delta}\psi_h) + s_{\E}^{\Delta}(u_h-\Pi_{\E}^{\Delta}u_h,\psi_h-\Pi_{\E}^{\Delta}\psi_h)\,,\\
a_{\E,h}^{\nabla}:V_h(P)\times V_h(P)\to \mathbb{R}\,, \qquad a_{\E,h}^{\nabla}(\sigma_h, v_h) 
&:= a_{\E}^{\nabla}(\Upsilon_{\E}^{\nabla} \sigma_h,\Upsilon_{\E}^{\nabla} v_h) + s_{\E}^{\nabla}(\sigma_h-\Upsilon_{\E}^{\nabla}\sigma_h, v_h-\Upsilon_{\E}^{\nabla} v_h)\,,\\
a_{\E,h}^{0}:V_h(P)\times \Psi_h(P)\to \mathbb{R}\,,\qquad a_{\E,h}^{0}(\sigma_h,\psi_h) &:= a_{\E}^{0}(\Upsilon_{\E}^{0} \sigma_h,\Pi_{\E}^{0}\psi_h)\,.
\end{align*}
Notice that such bi-linear forms are composed by two parts.
The first one where we have the corresponding continuous operator applied to a proper projection of the virtual element function.
The second one is a stabilisation term, i.e., 
$s^{\Delta,\E}(\cdot,\cdot )$ and  $s^{\nabla,\E}(\cdot,\cdot )$ 
are positive definite bi-linear form defined on $\Psi_h(P)\times \Psi_h(P)$ and $V_h(P)\times V_h(P)$, respectively, such that:
\begin{equation*}
\begin{array}{lcrcrl}
\alpha_{*}a_{\E}^{\Delta}(\psi_h,\psi_h)&\leq& s_{\E}^{\Delta}(\psi_h,\psi_h)&\leq & \alpha^{*}a_{\E}^{\Delta}(\psi_h,\psi_h) 
&\forall \psi_h\in \Psi_h(\E)\cap\ker(\Pi_{\E}^{\Delta})\,,\\[0.2em]
\beta_{*}a_{\E}^{\nabla}( v_h, v_h)&\leq& s_{\E}^{\nabla}( v_h, v_h)&\leq& \beta^{*}a_{\E}^{\nabla}( v_h, v_h)
&\forall  v_h\in V_h(\E)\cap\ker(\Upsilon_{\E}^{\nabla})\,
\end{array}    
\end{equation*}
where, $\alpha_{*},\alpha^{*},\beta_{*}$ and $\beta^{*}$ are positive constants independent of $h_P$.

We also define the discrete linear functional $F_h:V_h(P)\to \mathbb{R}$ given by 
\begin{align*}
F_{P,h}( v_h):=-\int_P \Upsilon_P^0 f  v_h.
\end{align*}

Then, the discrete global bi-linear forms and discrete global functional are given by the sum of each local contribution, i.e.,
\begin{align*}
&a_{h}^{\Delta}:\Psi_h(\E)\times \Psi_h(\E) \to \mathbb{R}; \qquad
a_{h}^{\Delta}(u_h,\psi_h):=\sum_{\E\in\Omega_h}a_{P,h}^{\Delta}(u_h,\psi_h)\,,\\
&a_{h}^{\nabla}:V_h(\E)\times V_h(\E) \to \mathbb{R}; \qquad
a_{h}^{\nabla}(\sigma_h, v_h):=\sum_{\E\in\Omega_h}a_{\E,h}^{\nabla}(\sigma_h, v_h)\,,\\
&a_{h}^{0}:V_h(\E)\times \Psi_h(\E) \to \mathbb{R}; \qquad
a_{h}^{0}(\sigma_h,\psi_h):=\sum_{\E\in\Omega_h}a_{\E,h}^{0}(\sigma_h,\psi_h)\,\\
&F_{h}:V_h(\E)\to  \mathbb{R}; \qquad \qquad \qquad
F_{h}( v_h):=\sum_{\E\in\Omega_h}F_{\E,h}( v_h)\,.
\end{align*}

In the next two proposition we collect two important results of VEM 
that guaranteed the well-posedness of the virtual element formulation of Problem~\ref{VarForContH2H1}.

\begin{proposition}[Consistency]\label{prop:consistency}
The following relations holds 
\begin{align*}
& a_{\E,h}^{\Delta}(q,\psi_h)=a_{\E}^{\Delta}(q,\psi_h)\qquad \forall q\in \mathbb{P}_{2}(\E) \quad \forall \psi_h\in \Psi_h(\E)\,,\\
& a_{\E,h}^{\nabla}(q, v_h)=a_{\E}^{\nabla}(q, v_h)\qquad \forall q\in \mathbb{P}_{1}(\E) \quad \forall  v_h\in V_h(\E)\,,\\
& a_{\E,h}^{0}(q,\psi_h)=a_{\E}^{0}(q,\psi_h)\qquad \forall q\in \mathbb{P}_{1}(\E) \quad \forall \psi_h\in \Psi_h(\E)\,,\\
&F_{P,h}( v_h)=F_P(\Upsilon_P^0 v_h)  \qquad \qquad \qquad \qquad  \forall  v_h\in V_h(\E)\,.
\end{align*}
\end{proposition}

\begin{proof}
The proof follows standard arguments in the VEM literature and it is already presented in the literature,
see, e.g.,~\cite{BDR2019C1Polyhedral}.
\end{proof}

\begin{proposition}[Stability]\label{prop:stability}
The following relations hold
\begin{equation*}
\begin{array}{lcrcrl}
\min\{1, \alpha_{*}\}a_{\E}^{\Delta}(\psi_h,\psi_h)&\leq& a_{\E,h}^{\Delta}(\psi_h,\psi_h)&\leq& \max\{1,\alpha^{*}\}a_{\E}^{\Delta}(\psi_h,\psi_h) 
&\forall \psi_h\in \Psi_h(\E)\cap\ker(\Pi_{\E}^{\Delta})\,,\\[0.2em]
\min\{1, \beta_{*}\}a_{\E}^{\nabla}( v_h, v_h)&\leq& a_{\E,h}^{\nabla}( v_h, v_h)&\leq& \max\{1,\beta^{*}\}a_{\E}^{\nabla}( v_h, v_h)
&\forall  v_h\in V_h(\E)\cap\ker(\Upsilon_{\E}^{\nabla})\,.
\end{array}    
\end{equation*}
\end{proposition}

\begin{proof}
The proof follows standard arguments in the VEM literature and it is already presented in the literature,
see, e.g.,~\cite{BDR2019C1Polyhedral}.
\end{proof}

Starting from Proposition~\ref{prop:consistency} and~\ref{prop:stability} we obtain the following lemma that is the discrete version of Lemma~\ref{lemboundcont}.
\begin{lemma}
The following estimates hold
\begin{subequations}
\begin{align}
|a_{h}^{\Delta}(u_{h},\psi_{h})|
& \lesssim  \max\{1,\alpha^{*}\} ||u_{h}||_{2,\Omega}||\psi_{h}||_{2,\Omega}
&\forall u_{h},\psi_{h}&\in \Psi_{h}\,,\label{Bound_aD_h}\\
|a^{\nabla}(u_{h},\psi_{h})|
& \lesssim\max\{1,\beta^{*}\}||u_{h}||_{1,\Omega}||\psi_{h}||_{1,\Omega}
&\forall u_{h},\psi_{h}&\in \Psi_{h}\,,\label{Bound_an_h}\\
|a_{h}^{0}(u_{h},\psi_{h})|
& \lesssim ||u_{h}||_{0,\Omega}||\psi_{h}||_{0,\Omega}
&\forall u_{h},\psi_{h}&\in \Psi_{h}\,,\label{Bound_a0_h}\\
|F_h( v_h)|
&\lesssim ||f ||_{-1,\Omega} || v_h ||_{1,\Omega}
&\forall  v_{h}&\in V_{h}\,,\label{BoundF_h}\\
a_{h}^{\Delta}(\psi_{h},\psi_{h})
& \geq  C_{P_2}\min\{1,\alpha_{*} \}  ||\psi_{h}||_{2,\Omega}^2
&\qquad\forall \psi_{h}&\in \Psi_{h}\,, \label{ellip_aD_h}\\
a_{h}^{\nabla}( v_{h}, v_{h})
& \geq  C_{P_1}\min\{1,\alpha_{*} \}  ||\psi_{h}||_{2,\Omega}^2
&\qquad\forall \psi_{h}&\in \Psi_{h}\,.\label{ellip_an_h}
\end{align}
\end{subequations}
\end{lemma} 

\begin{proof}
The proof of such result follows the same arguments of the proof of Lemma~\ref{lemboundcont} and 
it exploits Proposition~\ref{prop:consistency} and~\ref{prop:stability}.
\end{proof}

We are in a position to write the discrete virtual element variational formulation for problem \eqref{VarForContH2H1}. 

Find $\sigma_h\in V_h$ and $u_h\in \Psi_h$ such that 
\begin{equation}\label{VarForContH2H1_h}
a_h^{\Delta}(u_h,\psi_h) +a_h^{\nabla}(\sigma_h, v_h)+a_h^{0}(\sigma_h,\psi_h)=F_h( v_h)\quad  \forall  v_h \in V_h, \ \forall \psi_h\in \Psi_{h}.
\end{equation}

Now, the well-posedness of the  discrete variational formulation~\eqref{VarForContH2H1_h} follows with the same steps as the continous case. Indeed, we 
decompose \eqref{VarForContH2H1_h} in the  following two discrete problems.
\begin{problem}\label{ContProbH1_h}
Find $\sigma_h\in V_h$ such that 
\begin{align}
a_h^{\nabla}(\sigma_h, v_h)=F_h( v_h) \quad  \forall  v_h \in V_h.\nonumber
\end{align}
\end{problem}
\begin{problem}\label{ContProbH2_h}
Find $u_h\in  \Psi_h$ such that 
\begin{align}
\displaystyle a_h^{\Delta}(u_h,\psi_h) =-a_h^{0}(\sigma_h,\psi_h) \quad  \forall \psi_h\in \Psi_{h}.\nonumber
\end{align}
\end{problem}

Note that Lax-Milgram theorem and estimates \eqref{Bound_an_h}, \eqref{ellip_an_h}, and \eqref{BoundF_h} imply the well-posedness of Problem~\ref{ContProbH1_h}. Moreover, we apply once again the Lax-Milgram theorem joint with the estimates  \eqref{Bound_aD_h}, \eqref{ellip_aD_h}, \eqref{Bound_a0_h}, and the uniqueness of $\sigma_h\in V_h$ on Problem~\ref{ContProbH1_h} to obtain the unique solution $u_h\in \Psi_{h}$ of Problem~\ref{ContProbH2_h}. Then, we have proved the following result.
\begin{theorem}\label{UniqCont_h}
Given $f\in L^2(\Omega)$ there exists a unique $(\sigma_h,u_h)\in V_h\times \Psi_h$ such that the weak formulation \eqref{VarForContH2H1_h} holds true.
\end{theorem}

\section{Convergence and error estimates}
\label{SecConvErrEst}

In this section we will prove that $||\sigma-\sigma_h ||_{1,\Omega}+||u-u_h||_{2,\Omega}\to 0$ when $h$ goes to zero. With this end, we will present  some preliminary interpolation results well known in the literature of the virtual element methods which will be needed to prove convergence of the discrete problem \eqref{VarForContH2H1_h}.

\subsection{Some interpolation results}
We start with the following result on star-shaped polygons, 
which is derived by interpolation between Sobolev spaces (see for instance~\cite[Theorem I.1.4]{GR}).
We mention that this result has been stated
in~\cite[Proposition 4.2]{BBCMMR2013} for integer values
and follows from the classical Scott-Dupont theory, see~\cite{BS-2008}.
\begin{proposition}\label{InterpPoly}
	If the assumption {\bf A1} is satisfied, then for every $\psi\in H^{\delta}(\E)$ there exists
	$v_{\pi}\in\mathbb{P}_k(\E)$, $k\geq 0$ such that
	\begin{equation*}
	\vert \psi-\psi_{\pi}\vert_{\ell,\E}\lesssim  h_\E^{\delta-\ell}|\psi|_{\delta,\E}\quad 0\leq\delta\leq
	k+1, \ell=0,\ldots,[\delta],
	\end{equation*}
	with $[\delta]$ denoting the largest integer equal or smaller than $\delta \in {\mathbb R}$,
\end{proposition}
where we have used the following definition for the broken seminorm.
\begin{align*}
|\psi|_{\ell,h}^{2}&:=\sum_{\E\in \Omega_{h} }|\psi|_{\ell,\E}^{2} \quad
\forall \psi\in  L^2(\Omega): \psi|_P\in  H^{\ell}(P)\quad \qquad \ell=1,2.
\end{align*}
\begin{proposition}\label{InterpQh}
	Assume {\textbf{A1}} and {\textbf{A2}} are satisfied,
	let $v\in H^{1+t}(\Omega)$ with $t\in(0,1]$. Then, there exist $v_{I}\in V_h$ such that
	$$\Vert v-v_{I}\Vert_{1,\Omega}\lesssim   h^t\vert v \vert_{1+t,\Omega}.$$
\end{proposition}

\begin{proposition}\label{InterpVh}
Assume {\textbf{A1}} and {\textbf{A2}} are satisfied, 
let $\psi\in H^{\varepsilon}(\Omega)$ with $\varepsilon\in[2,3]$. Then, there exist $\psi_I\in \Psi_h$ such that 
\begin{equation*}
\Vert \psi-\psi_{I}\Vert_{\ell,\Omega}\lesssim
h^{\varepsilon-\ell}\Vert \psi\Vert_{\varepsilon,\Omega},\qquad \ell=0,1,2.
\end{equation*}
\end{proposition}

The proof for two last interpolation results on the virtual spaces $V_h$ and $\Psi_h$ can be checked in  \cite{BBCMMR2013,MRR2015} and \cite{ABSVsinum16,BM13}, respectively.

\subsection{A priori error estimates}
In this section we present the error estimates for the discrete scheme analyzed in this work. We start with the estimates for  $\sigma_h$.
\begin{lemma}
Let $\sigma  \in V$ and $\sigma_h\in V_h$ be the unique solutions of Problems~\ref{ContProbH1} and \ref{ContProbH1_h}, then the following result holds 
\begin{equation}
||\sigma - \sigma_h ||_{0,\Omega} + h
|\sigma - \sigma_h |_{1,\Omega}\lesssim h^2.
\label{addreg_H1_sigma}
\end{equation}
\end{lemma}
\begin{proof}
The proof has been established in \cite{HO_C0VEM_Polyhedral}.
%Let us consider $\sigma$ and $\sigma_h$ be the unique solutions of Problems~\ref{ContProbH1} and \ref{ContProbH1_h}, then from first Strang Lemma (see for instance \cite{}) we get:
%\begin{align}
%||\sigma - \sigma_h  ||_{1,\Omega} \leq C 
%\Bigg\{  
%&\inf\limits_{\tau_h \in V_h}
%\Bigg[
%||\sigma -\tau_h||_{1,\Omega} 
%+ \sup_{\substack{_h \in V_h\\ v_h\neq 0}}
%\sum\limits_{P\in \Omega_h} 
%\frac{
%\{a_{P,h}^{\nabla}(\tau_h,v_h) - a_{P}^{\nabla}(\tau_h,v_h) \} 
%}{||v_h||_{1,\Omega}}
%\Bigg]
%\nonumber\\
%+&\sup_{\substack{v_h \in V_h\\ v_h\neq 0}}
%\sum\limits_{P\in \Omega_h} 
%\frac{
%\{ F_h(v_h) - F(v_h) \}
%}{||v_h||_{1,\Omega}}
%\Bigg\}
%\label{strang1_u}
%\end{align}
\end{proof}

Now, the following result establishes the order of convergence for $u_h$.
\begin{theorem}
Let $(\sigma,u)\in V\times \Psi$ and $(\sigma_h,u_h)\in V_h\times \Psi_h$ be the unique solutions of Problems~\ref{ContProbH1}-\ref{ContProbH2} and \ref{ContProbH1_h}-\ref{ContProbH2_h}, respectively, then the following result holds
\begin{equation}
|u - u_h |_{2,\Omega}\lesssim h.
\label{sn2_uuh}
\end{equation}
\end{theorem}
\begin{proof}
Let us consider $u$ and $u_h$ be the unique solutions of Problems~\ref{ContProbH2} and \ref{ContProbH2_h}, then from first Strang Lemma (see for instance \cite[Theorem~4.1.1]{ciarlet}) we get:
\begin{align}
||u-u_h ||_{2,\Omega} \leq C 
\Bigg\{  
&\inf\limits_{\psi_h \in \Psi_h}
\Bigg[
||u-\psi_h||_{2,\Omega} 
+\sup_{\substack{\varphi_h \in \varphi_h\\ \varphi_h\neq 0}}
\sum\limits_{P\in \Omega_h}
\frac{
\{a_{P,h}^{\Delta}(\varphi_h,\psi_h) - a_{P}^{\Delta}(\varphi_h,\psi_h) \}
}{||\varphi_h||_{2,\Omega}} 
\Bigg]
\nonumber\\
+&\sup_{\substack{\varphi_h \in \varphi_h\\ \varphi_h\neq 0}}\sum\limits_{P\in \Omega_h}
\frac{
\{a_{P,h}^{0}(\sigma_h,\varphi_h) - a_{P}^{0}(\sigma,\varphi_h) \}
}{||\varphi_h||_{2,\Omega}}
\Bigg\}.
\label{strang1_u}
\end{align}
In what follows we will bound the three terms on the right hand side of \eqref{strang1_u}. Indeed, from Lemma~\ref{Lema_RegAd} and Proposition~\ref{InterpVh} we have that there exists $u_I\in \Psi_h$ such that 
\begin{equation}
 ||u-u_{I} ||_{2,\Omega} \leq C h ||u||_{3,\Omega}.   \label{bound_uI}
\end{equation}
Thus, we have from \eqref{bound_uI} and Lemma~\ref{Lema_RegAd} that
\begin{equation}
\inf\limits_{\psi_h \in \Psi_h}
||u-\psi_h||_{2,\Omega} \leq 
    ||u-u_{I} ||_{2,\Omega} \leq C h ||u||_{3,\Omega} \leq C h ||f||_{0,\Omega}. \label{bound1_sn2}
\end{equation}
Now, for the second term we use again Lemma~\ref{Lema_RegAd} and Proposition~\ref{InterpPoly} to chose $u_{\pi}\in L^2(\Omega)$ such that $u_{\pi}\in \mathbb{P}_2(P)\ \forall P\in \Omega_h$ and
\begin{equation}
    |u-u_{\pi} |_{2,P} \leq  C h |u|_{3,P}. \label{bound_upi}
\end{equation}
Then, we use the consistency and stability properties of the bilinear form $a_{P,h}^{\Delta}(\cdot,\cdot)$  (cf. Propositions~\ref{prop:consistency} and \ref{prop:stability}, respectively), and \eqref{bound_upi} to obtain
\begin{align}
&\sum\limits_{P\in \Omega_h}
\frac{
\{a_{P,h}^{\Delta}(\varphi_h,\psi_h) - a_{P}^{\Delta}(\varphi_h,\psi_h) \}
}{||\varphi_h||_{2,\Omega}}
=
\sum\limits_{P\in \Omega_h}
\frac{
\{a_{P,h}^{\Delta}(\varphi_h,\psi_h-u_{\pi}) - a_{P}^{\Delta}(\varphi_h,u_{\pi}-\psi_h) \}
}{||\varphi_h||_{2,\Omega}}\nonumber\\
&\sum\limits_{P\in \Omega_h} (1+\alpha^*)a_P^{\Delta}(u_{\pi}-\psi_h,u_{\pi}-\psi_h)^{1/2}\leq (1+\alpha^*)\Big(
\Big\{ \sum\limits_{P\in \Omega_h}|u-u_{\pi} |_{2,P}\Big\} + ||u-\psi_h||_{2,\Omega} \Big)\leq Ch||f||_{0,\Omega}\nonumber\\
&\lesssim h.\label{bound2_sn2}
\end{align}
To bound the third term we use the consistency and stability properties of the bilinear form $a_{P,h}^{0}(\cdot,\cdot)$ (cf. Propositions~\ref{prop:consistency} and \ref{prop:stability}, respectively) to get
\begin{align}
\sum\limits_{P\in \Omega_h}
\frac{
\{a_{P,h}^{0}(\sigma_h,\varphi_h) - a_{P}^{0}(\sigma,\varphi_h) \}
}{||\varphi_h||_{2,\Omega}}
&=
\sum\limits_{P\in \Omega_h}
\frac{
\{a_{P}^{0}(\Upsilon_P^0\sigma_h,\Pi_P^0\varphi_h) - a_{P}^{0}(\sigma,\varphi_h) \}
}{||\varphi_h||_{2,\Omega}}\nonumber\\
%\sum\limits_{P\in \Omega_h} \frac{
%\{a_{P}^{0}(\Upsilon_P^0\sigma_h,\varphi_h) - a_{P}^{0}(\sigma,\varphi_h) \}
%}{||\varphi_h||_{2,\Omega}}
&=
\sum\limits_{P\in \Omega_h}
\frac{
\{a_{P}^{0}(\Upsilon_P^0\sigma_h - \sigma,\varphi_h) \}
}{||\varphi_h||_{2,\Omega}}\nonumber\\
&=
\sum\limits_{P\in \Omega_h}
\frac{
\{a_{P}^{0}(\Upsilon_P^0\sigma_h - \sigma,\varphi_h-\Upsilon_P^0\varphi_h) \}
}{||\varphi_h||_{2,\Omega}}\nonumber\\
&\leq 
\sum\limits_{P\in \Omega_h}
\frac{
||\Upsilon_P^0\sigma_h -\sigma||_{0,P}||\varphi_h-\Upsilon_P^0\varphi_h ||_{0,P}
}{||\varphi_h||_{2,\Omega}}\nonumber\\
&\leq C h^2 ||\sigma||_{1,\Omega}
\lesssim h^2,
\label{bound3_sn2}
\end{align}
where we have employed the definitions of projectors $\Upsilon_P^0$ and $\Pi_P^0$ and \eqref{addreg_H1_sigma}.

Therefore, from estimates \eqref{bound1_sn2}, \eqref{bound2_sn2} and \eqref{bound3_sn2} the proof is finished.
\end{proof}

\subsubsection{Estimates in $H^1$ and $L^2$  for $u-u_h$}
In this section we will derive error bounds in $H^1$ and $L^2$ to estimate  $u-u_h$. To obtain that, we will resort to the standard duality arguments. More precisely, we will follow the steps applied in \cite{ChM-camwa} where the authors derive estimates in $H^1$ and $L^2$ for a bi-dimensional fourth order problem with clamped boundary conditions.

\begin{theorem}\label{Theo_SN1_u}
Let $(\sigma,u)\in V\times \Psi$ and $(\sigma_h,u_h)\in V_h\times \Psi_h$ be the unique solutions of Problems~\ref{ContProbH1}-\ref{ContProbH2} and \ref{ContProbH1_h}-\ref{ContProbH2_h}, respectively, then the following result holds
\begin{equation*}
|u - u_h |_{1,\Omega}\lesssim h^2.
\label{sn1_uuh}
\end{equation*}
\end{theorem}
\begin{proof}
Let us consider the following well posed fourth-order problem: find $\xi \in \Psi$ such that
\begin{equation}
    a^{\Delta}(\xi,\varphi)=a^{\nabla}(u-u_h,\varphi)\quad \forall \varphi \in \Psi.\label{VarFor_Bih_dual}
\end{equation}
It is well known from regularity results for biharmonic equation with right hand side in $H^{-1}(\Omega)$ on convex domain that the following estimate holds
\begin{equation}
||\xi||_{3,\Omega}\lesssim |u-u_h|_{1,\Omega}.
\label{Bound_Bih_dual}
\end{equation}
Then, from Proposition~\ref{InterpVh} and \eqref{Bound_Bih_dual} (twice) we have that there exists $\xi_I\in \Psi_h$ such that 
\begin{equation}
||\xi-\xi_I ||_{2,\Omega}\lesssim h ||\xi ||_{3, \Omega}    
\lesssim h |u-u_h|_{1,\Omega}\label{bound1_dual_Bih}.
\end{equation}
Now, by testing the variational formulation \eqref{VarFor_Bih_dual} with $\varphi=u-u_h\in \Psi$, adding and subtracting the term $\xi_I$, and using Problems~\ref{ContProbH2} and \ref{ContProbH2_h} we obtain 
\begin{align}
|u-u_h|_{1,\Omega}^2
&\lesssim a^{\Delta}(\xi,u-u_h) = a^{\Delta}(\xi-\xi_I,u-u_h) + a^{\Delta}(\xi_I,u-u_h) \nonumber\\
&=a^{\Delta}(\xi-\xi_I,u-u_h) + a^{\Delta}(u,\xi_I) - a^{\Delta}(u_h,\xi_I) \nonumber\\
&=a^{\Delta}(\xi-\xi_I,u-u_h) - a^{0}(\sigma,\xi_I) - a^{\Delta}(u_h,\xi_I) \nonumber\\
&=\underbrace{a^{\Delta}(\xi-\xi_I,u-u_h)}_{T_1} +\underbrace{\{a_h^0(\sigma_h,\xi_I) - a^0(\sigma,\xi_I) \}}_{T_2} 
+\underbrace{\{a_{h}^{\Delta}(u_h,\xi_I)- a^{\Delta}(u_h,\xi_I) \}}_{T_3}. \label{Sn1_T1T2T3}
\end{align}
Now, we will bound the three terms on the right hand side of \eqref{Sn1_T1T2T3}. Indeed, for the first term we apply \eqref{Bound_aDCont}, \eqref{bound1_dual_Bih} and \eqref{sn2_uuh} to get
\begin{equation}
    T_1\leq |\xi - \xi_I|_{2,\Omega}|u-u_h|_{2,\Omega}
    \leq ||\xi - \xi_I||_{2,\Omega}|u-u_h|_{2,\Omega}
    \lesssim h^2|u-u_h|_{1,\Omega}.\label{boundT1_sn1}
\end{equation}
To bound $T_2$ we follow the same steps as those applied to obtain \eqref{bound3_sn2} and we can deduce that
\begin{equation}
T_2\lesssim h^2|u-u_h|_{1,\Omega}.\label{boundT2_sn1}
\end{equation}
Next, to bound the term $T_3$ we use Lemma~\ref{Lema_RegAd},  \eqref{Bound_Bih_dual}, and Proposition~\ref{InterpPoly} to chose $u_{\pi},\xi_{\pi}\in \mathbb{P}_2(P)$ such that 
\begin{equation}
    |u-u_{\pi}|_{2,P}\leq h|u|_{3,P}
    \quad \mbox{and} \quad 
    |\xi-\xi_{\pi}|_{2,P}\leq h|\xi|_{3,P}.
    \label{u_pi_xi_pi}
\end{equation}
Then, we add and subtract the terms $u_{\pi}$ and $\xi_{\pi}$, and use the consistency property of the bilinear form $a_{h,P}^{\Delta}(\cdot,\cdot)$ (cf. Proposition~\ref{prop:consistency}) to have
\begin{align}
T_3
&=\sum\limits_{P\in \Omega_h}\{  a_{P,h}^{\Delta}(u_h,\xi_I)- a_{P}^{\Delta}(u_h,\xi_I)\}
=\sum\limits_{P\in \Omega_h} \{ a_{P,h}^{\Delta}(u_h-u_{\pi},\xi_I-\xi_{\pi})+ a_{P}^{\Delta}(u_{\pi}-u_h,\xi_I-\xi_{\pi})\}.\nonumber
\end{align}
Next, in the equality above we apply \eqref{Bound_aDCont}, \eqref{Bound_aD_h}, add and subtract the terms $u$ and $\xi$, apply \eqref{u_pi_xi_pi}, \eqref{sn2_uuh} and \eqref{bound1_dual_Bih} to deduce
\begin{equation}
T_3\lesssim h^2 |u-u_h|_{1,\Omega}. \label{boundT3_sn1}  
\end{equation}
Therefore, from \eqref{Sn1_T1T2T3}, \eqref{boundT1_sn1}, \eqref{boundT2_sn1} and \eqref{boundT3_sn1} we have proved the result.
\end{proof}

\begin{theorem}
Let $(\sigma,u)\in V\times \Psi$ and $(\sigma_h,u_h)\in V_h\times \Psi_h$ be the unique solutions of Problems~\ref{ContProbH1}-\ref{ContProbH2} and \ref{ContProbH1_h}-\ref{ContProbH2_h}, respectively, then the following result holds
\begin{equation*}
||u - u_h ||_{0,\Omega}\lesssim h^2.
\label{sn0_uuh}
\end{equation*}
\end{theorem}
\begin{proof}
Since $u,u_h\in \Psi \subset V$ we have from the definition of $||\cdot ||_{1,\Omega}$, Poincar\'e inequality (see for instance \cite{ciarlet}) and Theorem~\ref{Theo_SN1_u} the following estimates
\begin{equation*}
||u-u_h||_{0,\Omega}
\leq 
||u-u_h||_{1,\Omega} \lesssim |u-u_h|_{1,\Omega} \lesssim h^2,  
\end{equation*}
which proves the theorem.
\end{proof}

\section{Numerical examples}
\label{SEC:NumericalResults}

In this section we are going to give the numerical evidence about the result obtained in the previous sections. 
To achieve this goal we verify that the following error indicators:
\begin{itemize}
\item $H^2$-seminorm error on $u_h$
$$
\sqrt{\sum_{P\in \Omega_h}|\Pi_{\E}^{\Delta} u_h-u|_{2,P}^2}\approx h^2\,,
$$
\item $H^1$-seminorm error on $u_h$
$$
\sqrt{\sum_{P\in \Omega_h}|\Pi_{\E}^{\nabla} u_h-u|_{1,P}^2}\approx h^1\,,
$$
\item $L^2$ norm error on $u_h$
$$
\sqrt{\sum_{P\in \Omega_h}\|\Pi_{\E}^0 u_h-u\|_{2,P}^2}\approx h^2\,,
$$
\item $H^1$-seminorm error on  $\sigma_h$
$$
\sqrt{\sum_{P\in \Omega_h}|\Pi_{\E}^{\nabla} \sigma_h-\sigma|_{1,P}^2}\approx h^1\,,
$$
\item $L^2$ norm error on $\sigma_h$
$$
\sqrt{\sum_{P\in \Omega_h}\|\Pi_{\E}^0 \sigma_h-\sigma\|_{2,P}^2}\approx h^2\,,
$$
\end{itemize}
such error indicators consider both the problem variable $u_h$ 
and the auxiliary variable $\sigma_h$.

Notice that, since we are dealing with virtual functions, 
we have to use proper projection operators to compute such quantities.
Moreover, such projection operators do depend on the error we are computing.
For instance, to compute the error in the $H^2$-seminorm, 
we have to take the $H^2$ projection operator $\Pi_{\E}^\Delta$~Equation~\eqref{eqn:deltaProjDef}.
However, 
to compute the $L^2$ norm error we may use $\Pi_{\E}^\Delta$,
but for the VEM approximation degree we are considering 
the projection $\Pi_{\E}^\Delta$ coincides with $\Pi_{\E}^0$, see Remark~\ref{rem:coinPiDelatPiZero}.

In the subsequent examples we take the unit cube as domain $\Omega$.
We discretize it in four different ways:
\begin{itemize}
    \item \texttt{tetra}: a Delaunay tetrahedral mesh built via \texttt{tetgen}~\cite{si2015tetgen};
    \item \texttt{cube}: a structured mesh composed by cubes;
    \item \texttt{nine}: a structured tesselation of a cube with polyhedrons composed by nine faces;
    \item \texttt{voro}: a Voronoi mesh optimize via a Llyod algorithm~\cite{du1999centroidal} obtained by the library \texttt{voro++}~\cite{rycroft2009voro++}.
\end{itemize}
In Figure~\ref{fig:meshes} we show an example of each of these discretizations.
As you can notice from these pictures the meshes taken into account have an increasing complexity.
Indeed, the first two meshes are standard since they are composed by tetrahedrons and structured cubes.
The mesh \texttt{nine} is composed by a stencil of nine polyhedron with a regular shape repeated to cover the whole domain $\Omega$.
Finally the mesh \texttt{voro} is a Voronoi mesh that presents a non uniform distribution of faces and edges. 
Indeed, a polyhedron may have very small faces adjacent to larger ones and
it can also have very tiny edges near by larger ones.

\begin{figure}[!htb]
\centering
\begin{tabular}{ccc}
 \texttt{tetra} &&\texttt{cube} \\[1em]
\includegraphics[width=0.4\textwidth]{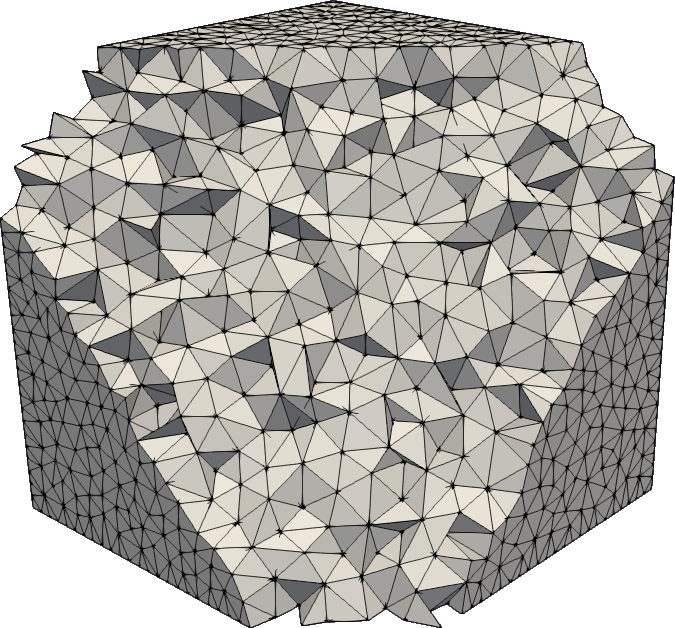} &\phantom{mm}&
\includegraphics[width=0.4\textwidth]{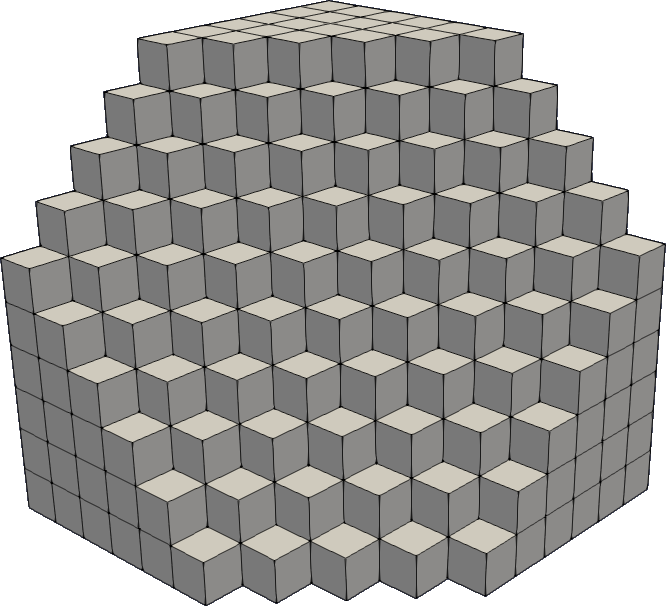} \\[1em]
\texttt{nine} &&\texttt{voro} \\[1em]
\includegraphics[width=0.4\textwidth]{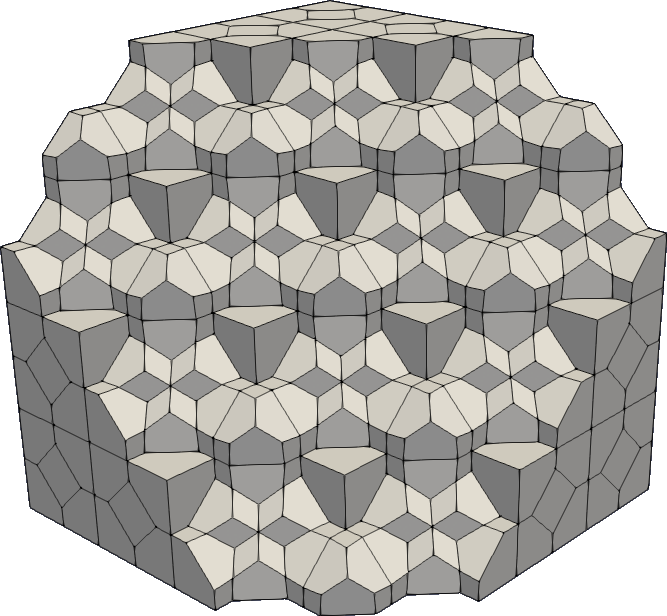} &\phantom{mm}&
\includegraphics[width=0.4\textwidth]{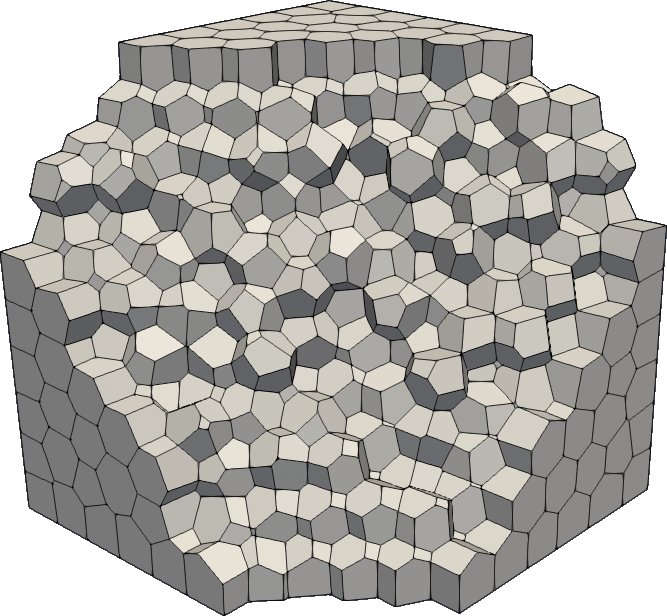} \\
\end{tabular}
\caption{Mesh types used for the numerical experiments.}
\label{fig:meshes}
\end{figure}

We associate with each mesh a mesh size $h$ defined as
$$
h = \frac{1}{N_P}\sum_{P\in \Omega_h} h_P\,,
$$
where $N_P$ is the number of polyhedrons of $\Omega_h$
and $h_P$ is the diameter of the polyhedron $P$.

Then, to compute the convergence lines,
we build a sequence of four meshes decreasing the mesh-size 
for each type of mesh taken into account.
%In Figure~\ref{fig:meshInc} we show an example of this family of meshes discretized via the \texttt{voro} strategy.

%\begin{figure}[!htb]
%\centering
%\begin{tabular}{cccc}
%\texttt{mesh 1} &\texttt{mesh 2} &\texttt{mesh 3} &\texttt{mesh 4} \\[1em]
%\includegraphics[width=0.21\textwidth]{voro1.png} &
%\includegraphics[width=0.21\textwidth]{voro2.png} &
%\includegraphics[width=0.21\textwidth]{voro3.png} &
%\includegraphics[width=0.21\textwidth]{voro4.png} \\
%\end{tabular}
%\caption{The sequence of \texttt{voro} meshes used.}
%\label{fig:meshInc}
%\end{figure}

\subsection{Patch test}

It is well-known that the virtual element approach passes the so-called patch test.
Indeed, consider a partial differential equation whose solution is a polynomial of degree $k$.
If the virtual element space we are using contains polynomials of degree $k$,
we recover exact solution up to machine precision.

In this paper the functional spaces introduced for the variable $u$, Equation~\eqref{eqn:fullC1Space},
contains polynomial of degree 2,
while the space of the auxiliary variable $\sigma$, Equation~\eqref{eqn:fullC0Space},
contains  polynomial of degree 1.
As a consequence, the proposed method will be ``excact'' if the solution is a polynomial of degree lower than 2,
while we have to recover the error trend for higher degrees.

To numerically verify this fact,
we consider a sixth-order elliptic problem whose right hand side and boundary conditions are set 
in such a way that the exact solution is the polynomial
\begin{equation}
u(x,\,y,\,z) = (x+y+z)^k\,.
\label{eqn:exactSolUPatch}
\end{equation}
Starting from this definition of $u$, also the auxiliary variable $\sigma$ is a polynomial and 
it has the following form
\begin{equation}
\sigma(x,\,y,\,z) = 
\begin{cases}
-9k(k-1)(k-2)(k-3)(x+y+z)^{k-4} &\text{if }k\geq 4\\
0 &\text{otherwise}\\
\end{cases}
\,.
\label{eqn:exactSolSigmaPatch}
\end{equation}
In both Equation~\eqref{eqn:exactSolUPatch} and~\eqref{eqn:exactSolSigmaPatch} $k$ is an integer parameter
that we can set to get a polynomial of a specific degree.

In this example we will show only the convergence lines for the mesh types $\texttt{tetra}$ and $\texttt{nine}$,
since we got similar results for the other two types of meshes.
In Figure~\ref{fig:exe1PatchU} we show the convergence rates of the $H^2$-seminorm and the $L^2$-norm errors.
In these graphs we observe that we get the solution up to machine precision when $k=1$ and $2$,
then for $k\geq 3$ we have the expected convergence rate.
This fact perfectly matches what we expected.
Indeed, since we have a virtual element space that contains polynomials of degree 2,
we are exact if the solution of the PDE is a polynomial of degree lower or equal to 2. 

\begin{figure}[!htb]
\centering
\begin{tabular}{cc}
\texttt{tetra} &\texttt{nine}\\
\includegraphics[width=0.44\textwidth]{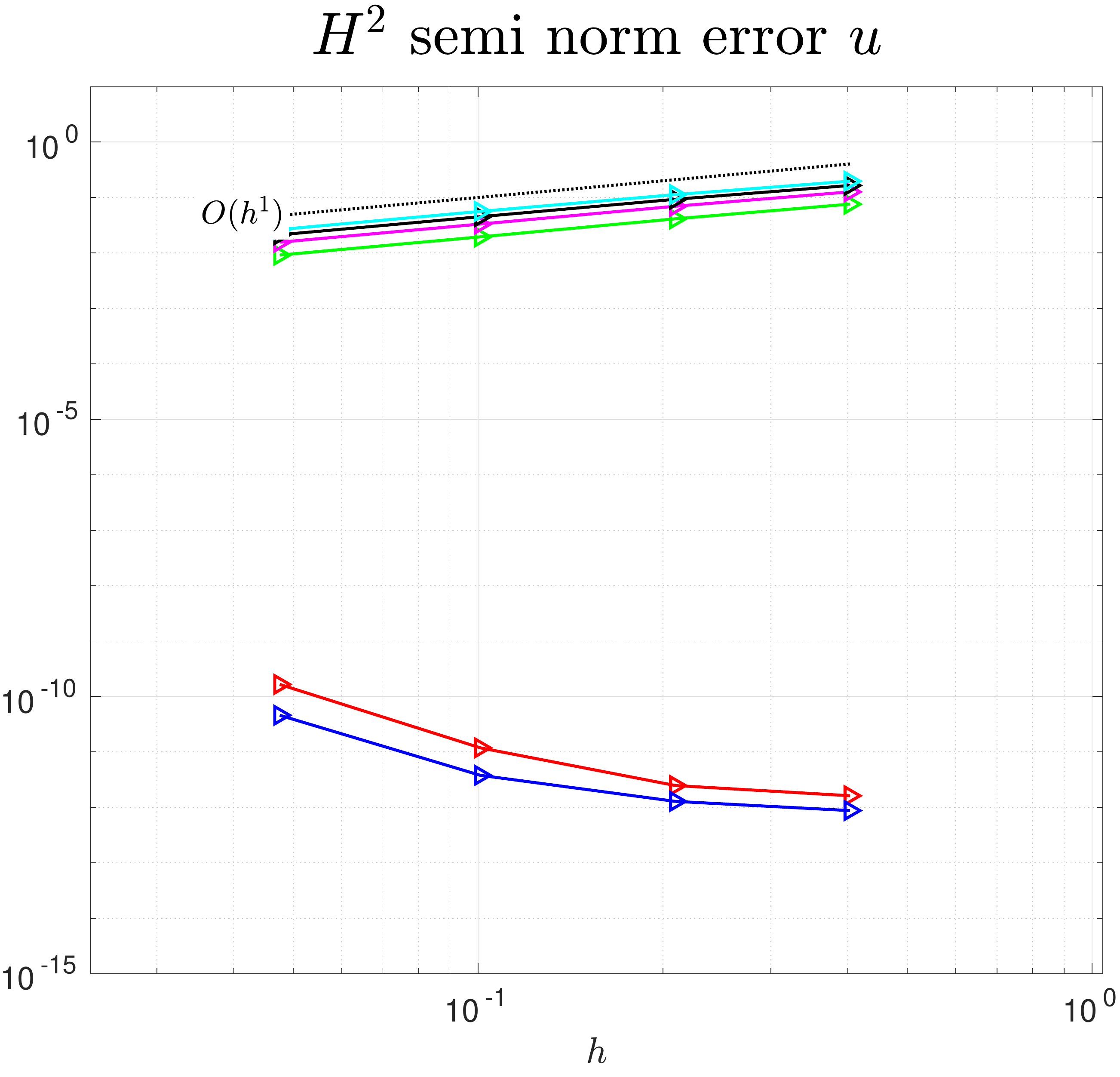}&
\includegraphics[width=0.44\textwidth]{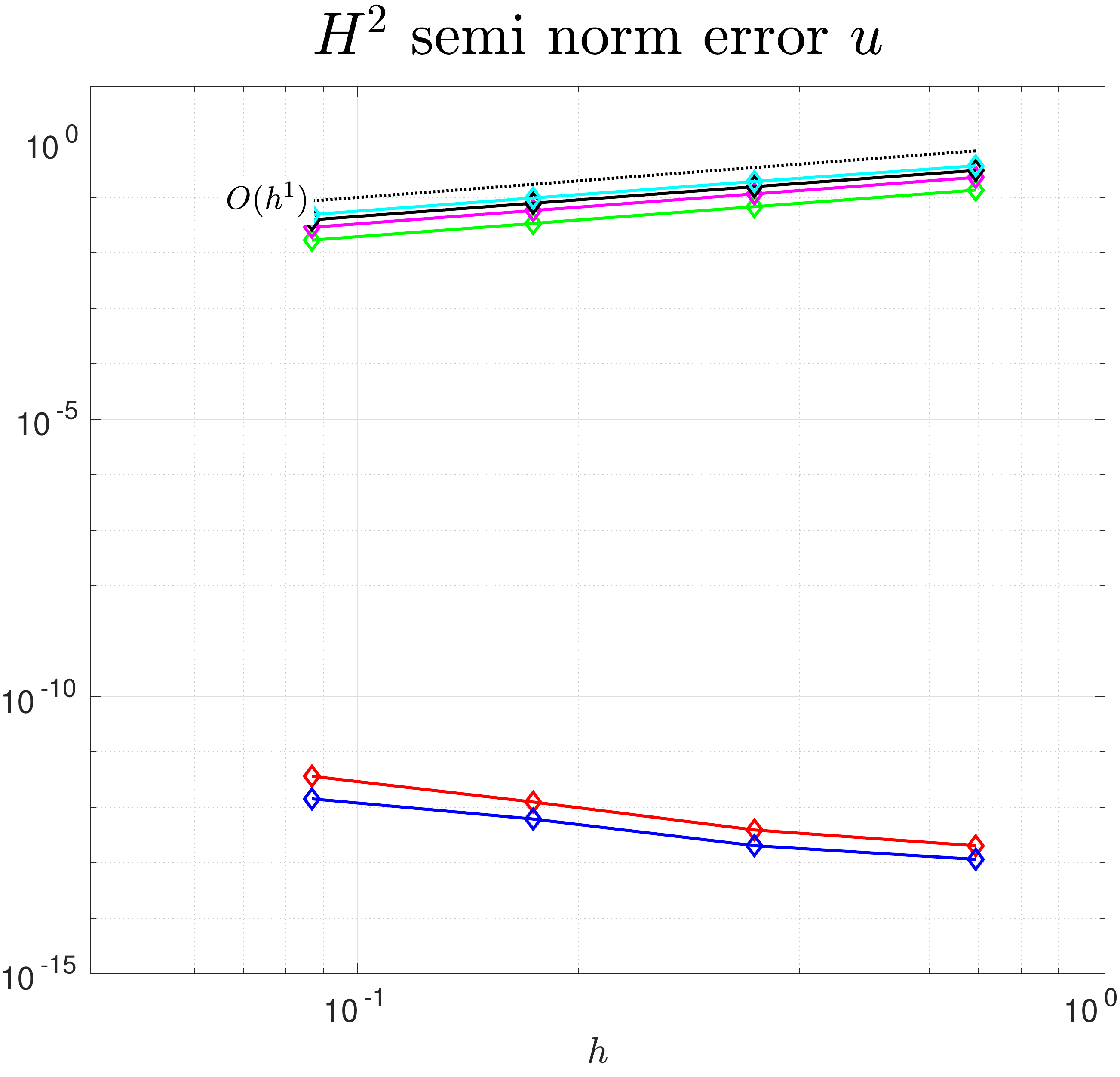}\\
\includegraphics[width=0.44\textwidth]{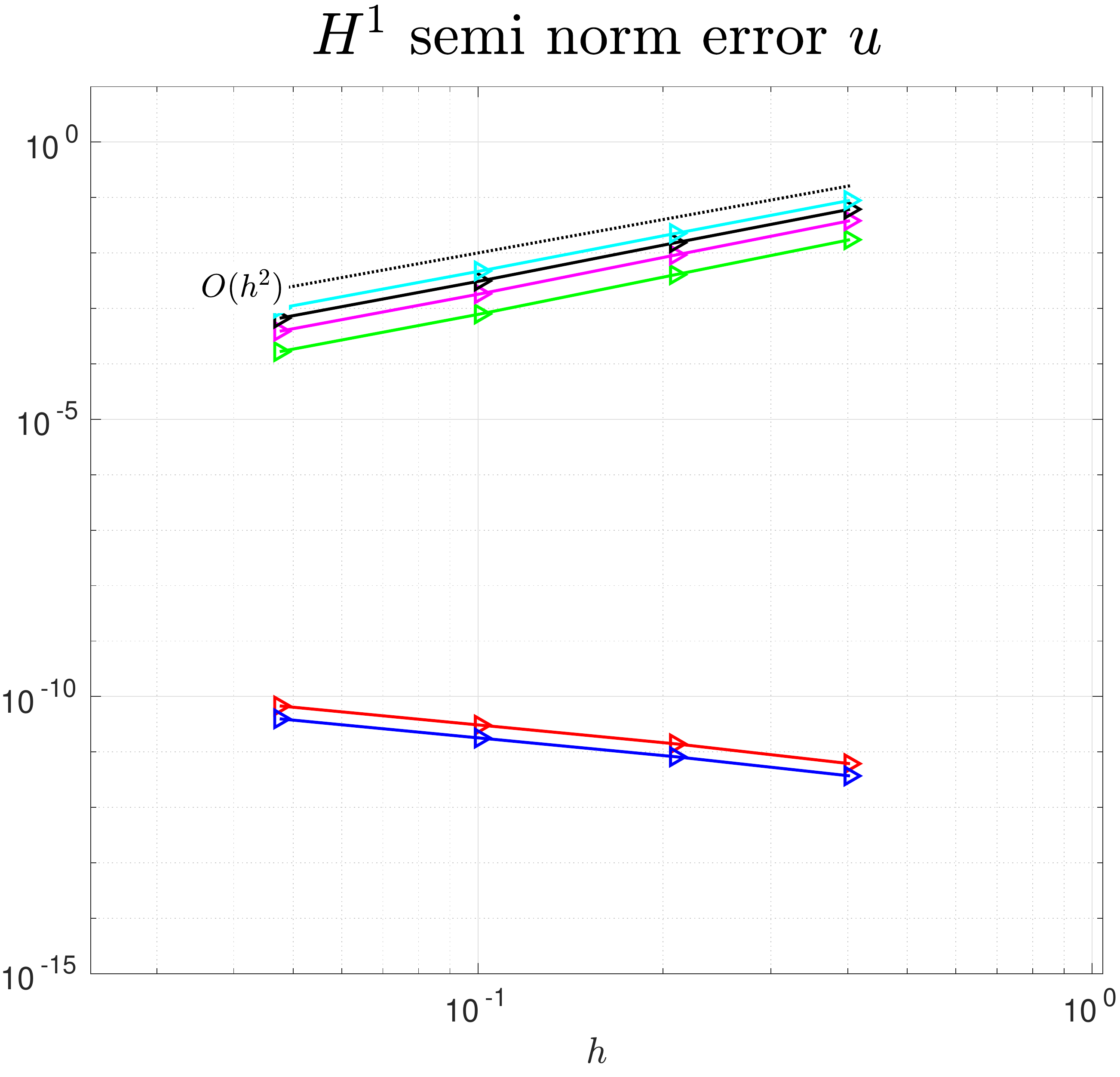}&
\includegraphics[width=0.44\textwidth]{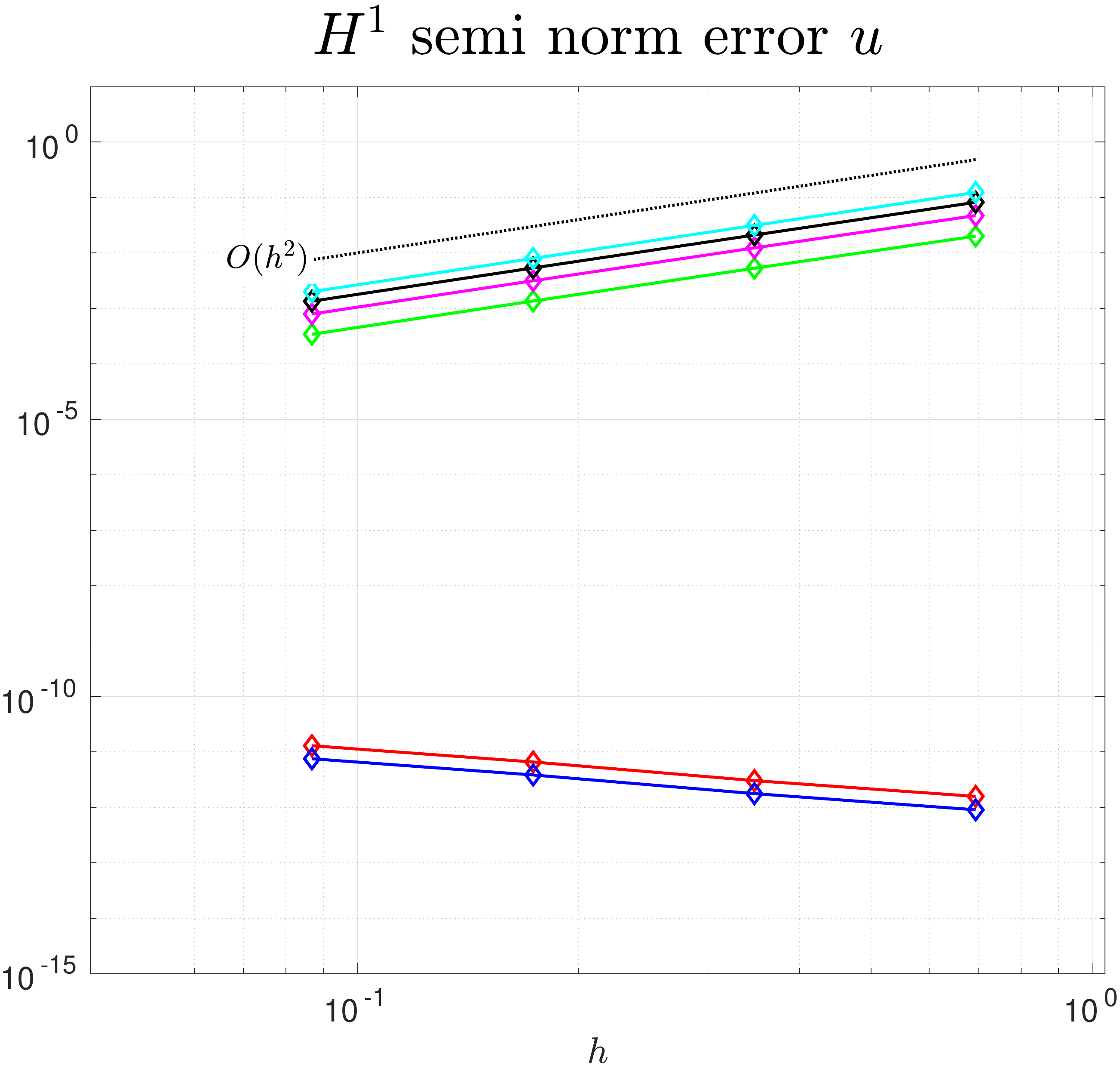}\\
\includegraphics[width=0.44\textwidth]{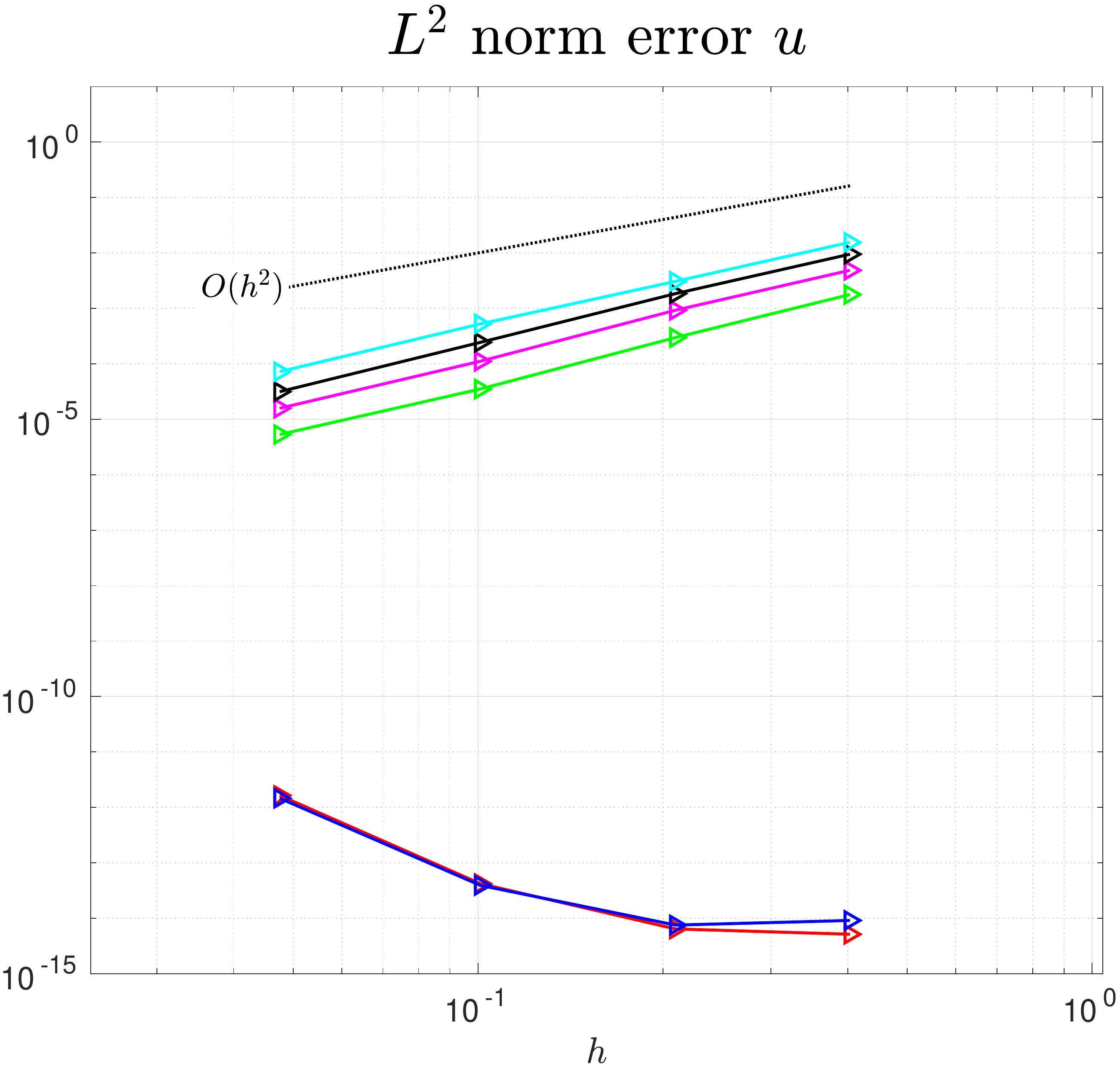}&
\includegraphics[width=0.44\textwidth]{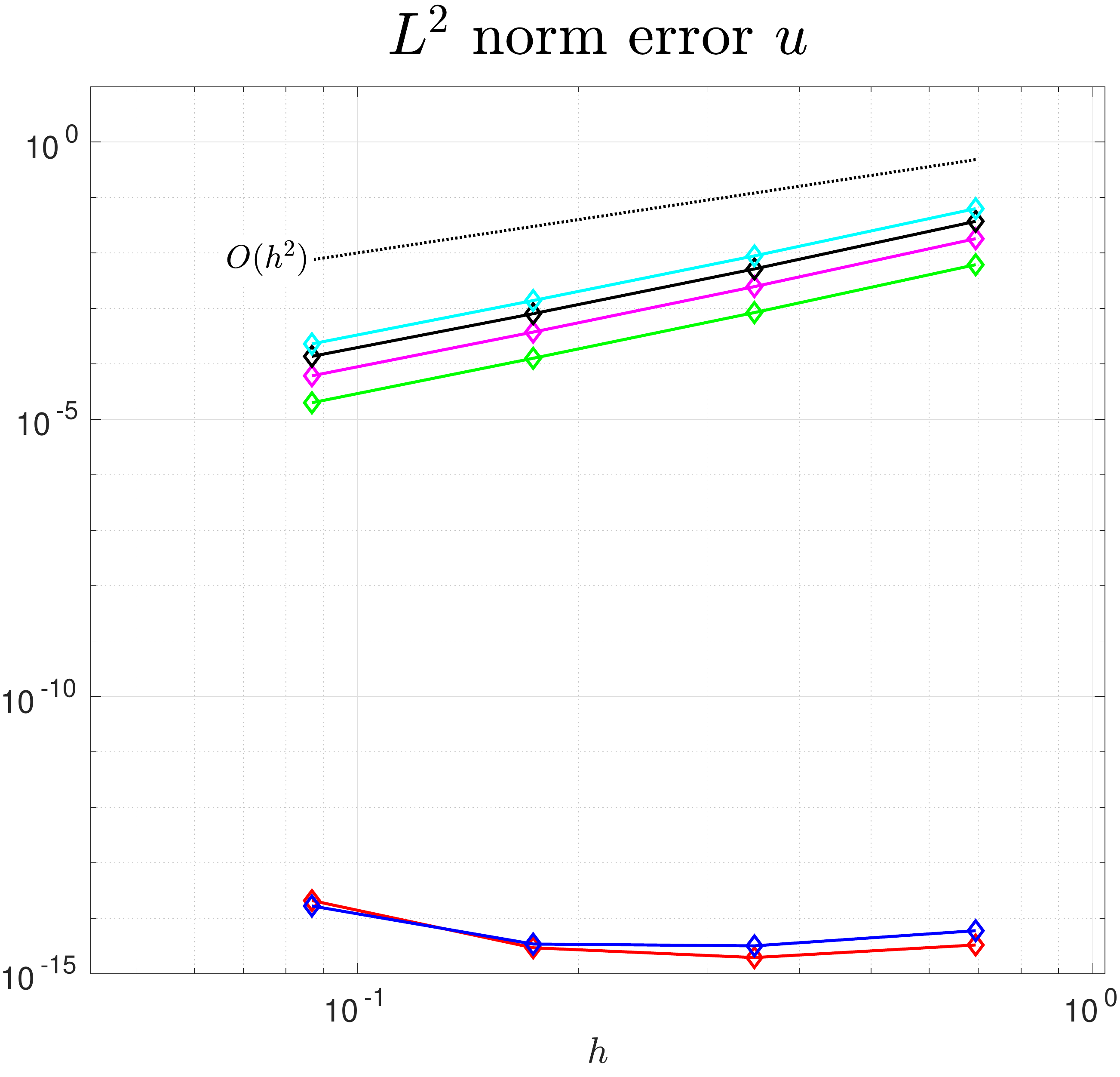}\\
\multicolumn{2}{c}{\includegraphics[width=0.60\textwidth]{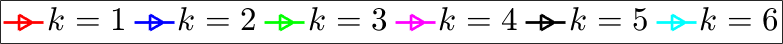}}
\end{tabular}
\caption{Patch test: convergence rate of $u$ for the \texttt{tetra} and \texttt{nine} meshes.}
\label{fig:exe1PatchU}
\end{figure}

Moreover, we observe that the convergence lines associated with $k=1$ and $2$ slowly increase.
This fact is probably due to machine algebra effect. 
Indeed, the bigger the linear system is,
the more truncation errors propagate so 
we get larger errors.

In Figure~\ref{fig:exe1PatchS} we show the behaviour of the error of auxiliary variable.
Here we got the exact solution up to degree 5.
It seems strange since the virtual element space of $\sigma_h$ contains polynomial of degree 1.
However it is definitely correct. 
Indeed, when the exact solution $u$ is a polynomial of degree $k\leq 5$,
its bi-Laplacian is a polynomial of degree lower or equal to 1.
However, if we consider $k=6$,
$\sigma$ is a polynomial of degree $2$,
it is not content in the virtual element gspace so 
we get the expected convergence rates.
Finally, also in this case we observe 
that the convergence lines associated with $k\leq 5$ slowly increase 
due to numerical linear algebra effect.

\begin{figure}[!htb]
\centering
\begin{tabular}{cc}
\texttt{tetra} &\texttt{nine}\\
\includegraphics[width=0.44\textwidth]{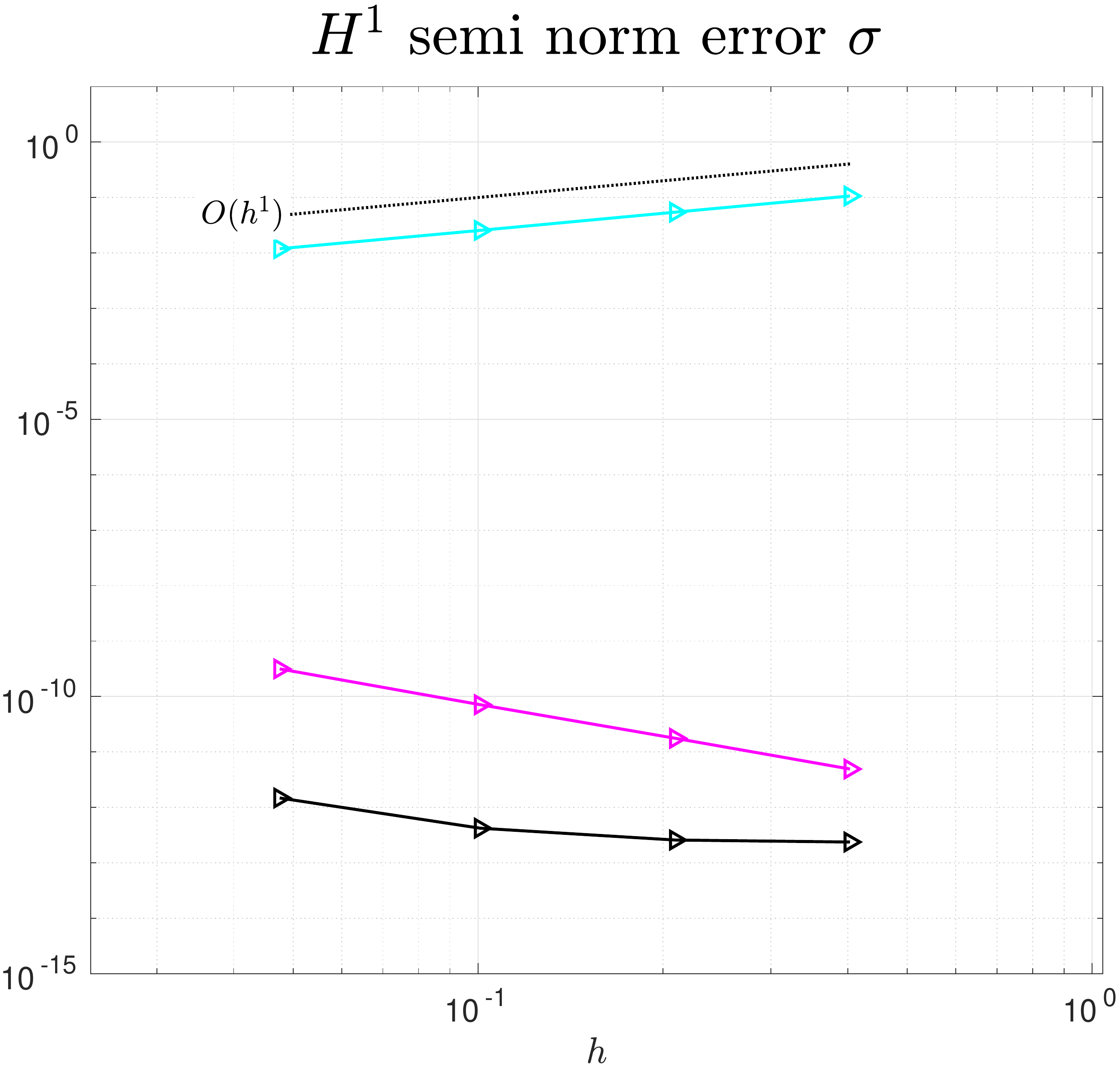}&
\includegraphics[width=0.44\textwidth]{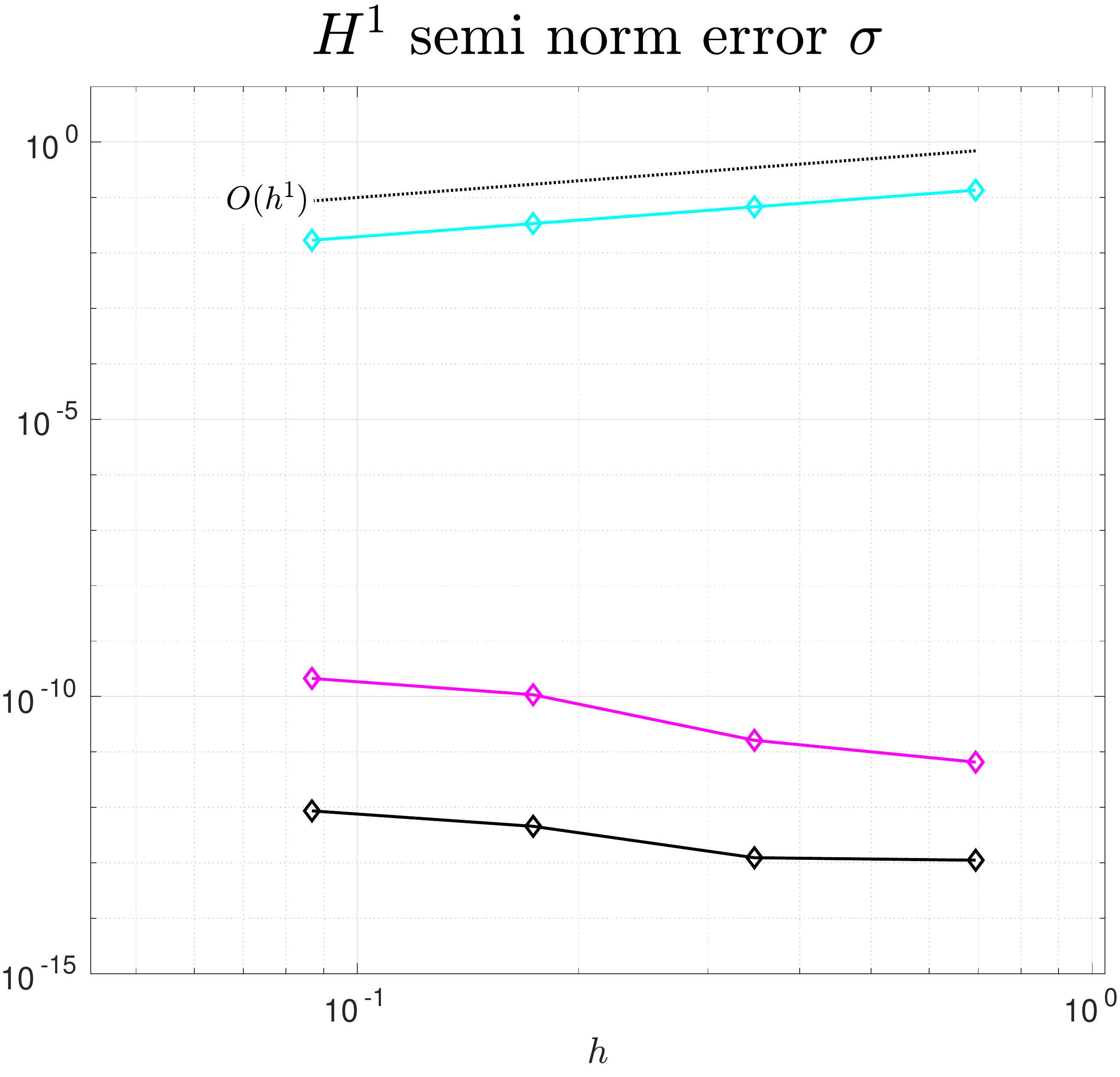}\\
\includegraphics[width=0.44\textwidth]{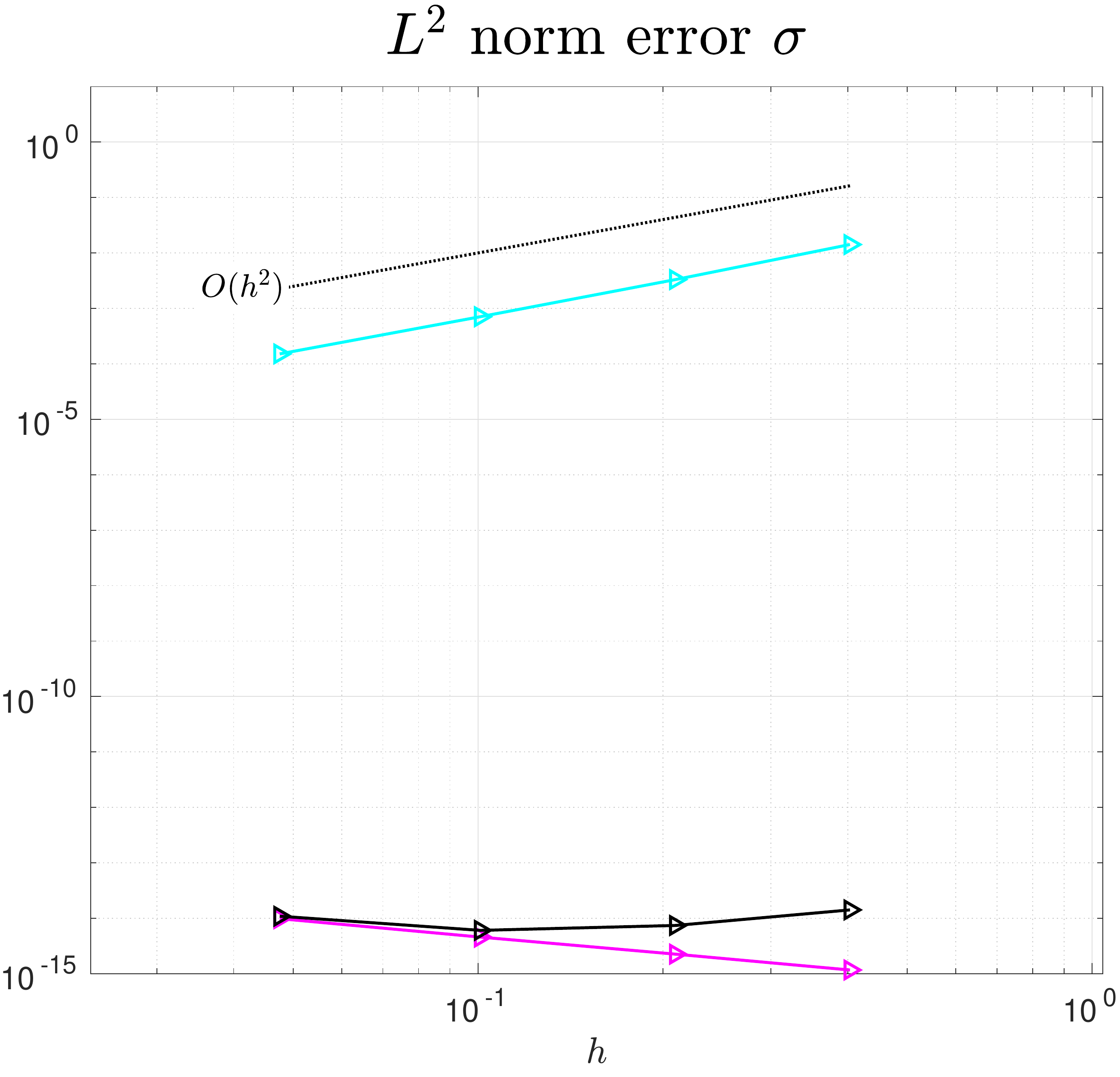}&
\includegraphics[width=0.44\textwidth]{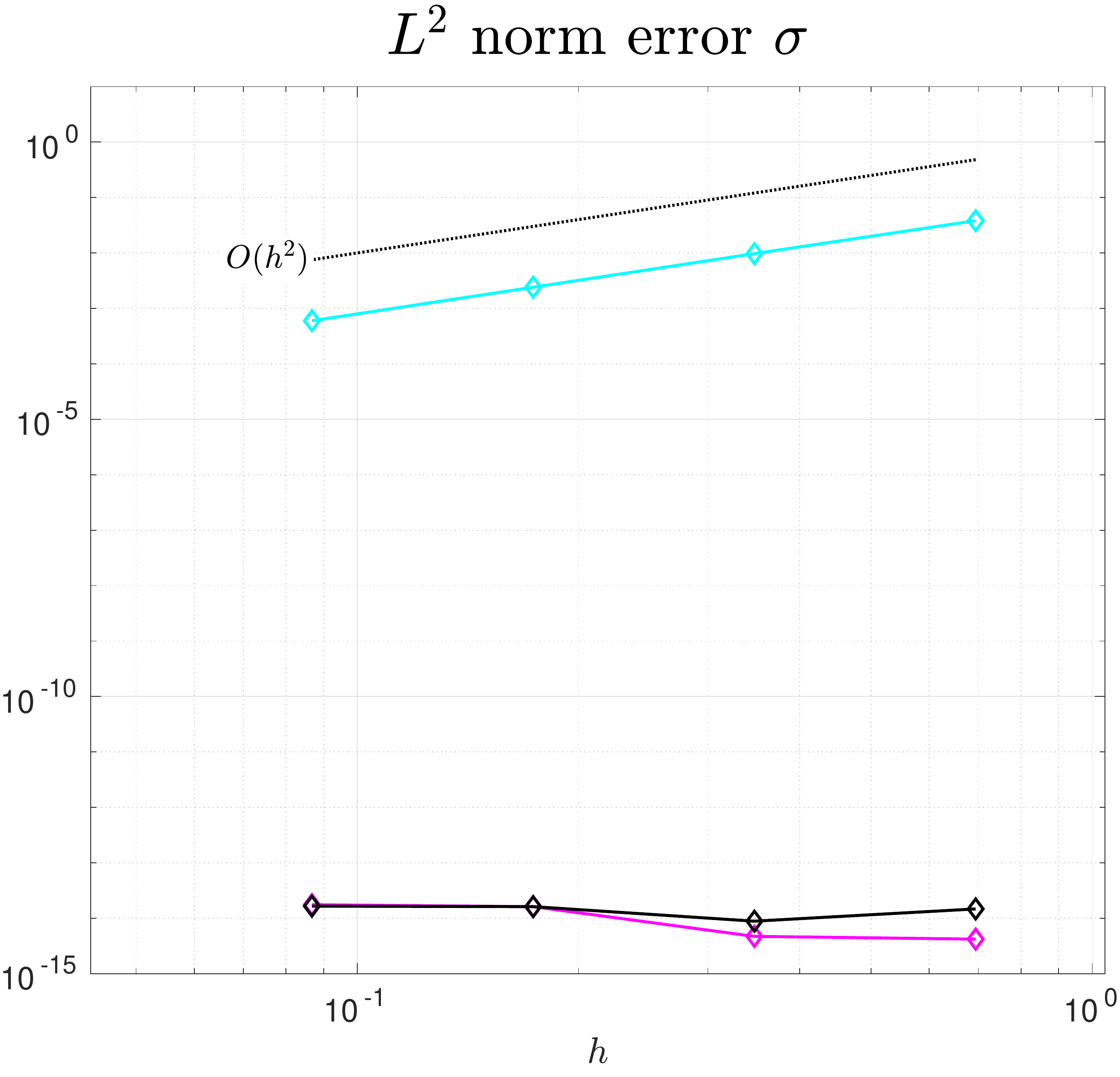}\\
\multicolumn{2}{c}{\includegraphics[width=0.60\textwidth]{legend.png}}
\end{tabular}
\caption{Patch test: convergence rate of $\sigma$ for the \texttt{tetra} and \texttt{nine} meshes.}
\label{fig:exe1PatchS}
\end{figure}

\subsection{Convergence Analysis}

In this subsection we numerically verify 
the error trends predicted by the theory.
To achieve this goal we take the sixth-order elliptic problem and 
we set the right hand side and the boundary condition in such a way 
that the exact solution is
$$
u(x,\,y,\,z) = \sin(\pi x)\,\sin(\pi y)\,\sin(\pi z)\,.
$$
In Figure~\ref{fig:exe2u} we show the convergence lines 
for $H^2$ seminorm and $L^2$ norm errors for each mesh type taken into consideration.
From these data we observe that we get the predicted rate for both the errors.
More specifically we got a linear decay of the $H^2$ seminorm erros and 
a quadratic one in the $L^2$ norm.
Moreover, the convergence lines are really close to each other.
Since for each refinement level each mesh type has a comparable mesh size,
the fact that the error is approximately the same  underlines the robustness of the proposed VEM scheme with respect to elements shape.
Similar considerations can be done for the trend of the $H^1$ seminorm error,
but we do not show such graph for brevity.
\begin{figure}[!htb]
\centering
\begin{tabular}{cc}
\includegraphics[width=0.44\textwidth]{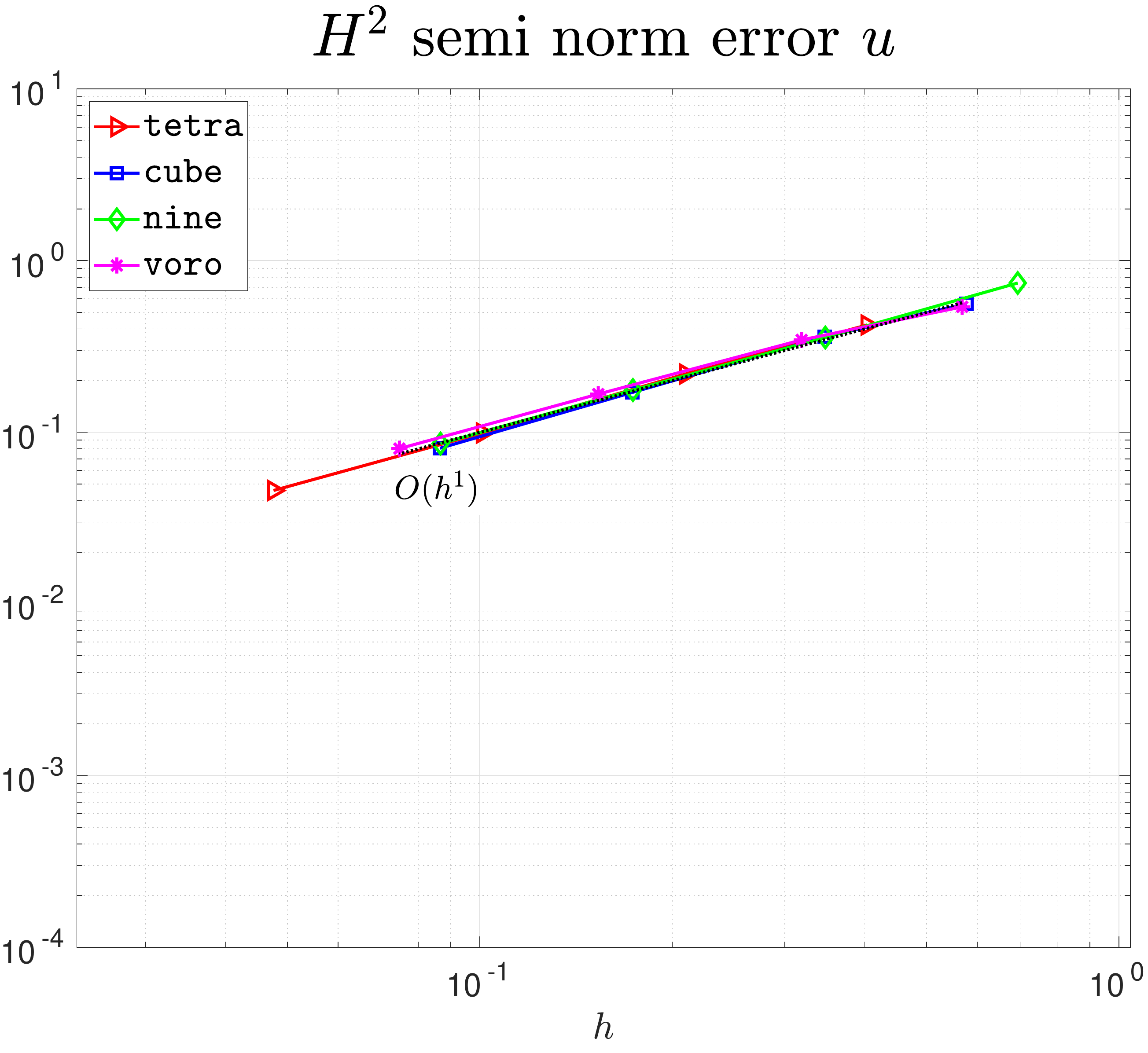}&
\includegraphics[width=0.44\textwidth]{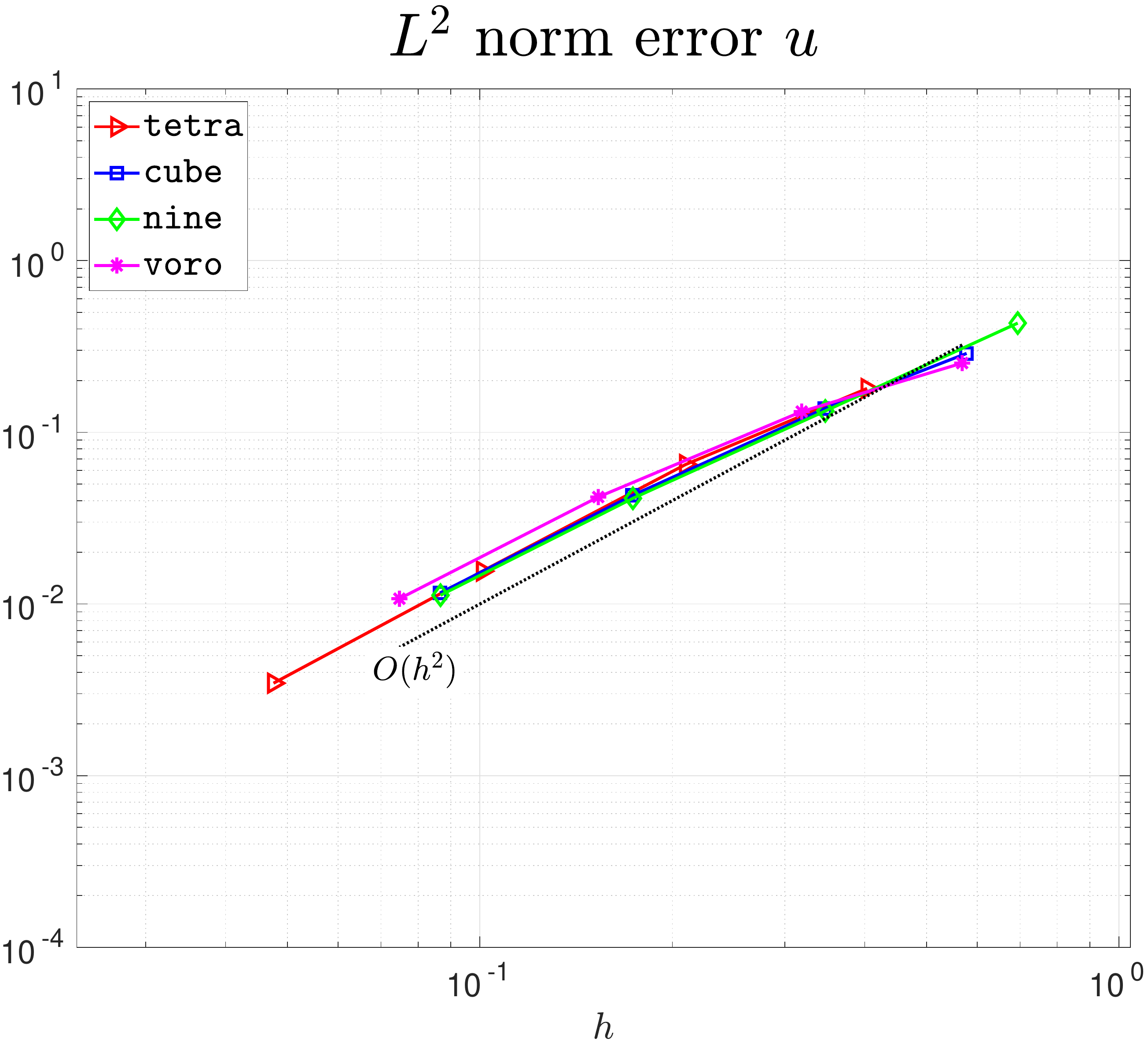}\\
%\multicolumn{2}{c}{\includegraphics[width=0.44\textwidth]{exe2uH1-eps-converted-to.pdf}}\\
\end{tabular}
\caption{Convergence Analysis: convergence rate of $u$ each mesh type.}
\label{fig:exe2u}
\end{figure}

In Figure~\ref{fig:exe2sigma} we show the error trend of the auxiliary variable $\sigma$.
Also in this case the convergence rate is the expected one:
$O(h)$ and $O(h^2)$ in the $H^1$ seminorm and in the $L^2$ norm, respectively.
Moreover, also in this case all the convergence lines are close one to each other.
As a consequence we can infer that the proposed method is robust with respect to element shape.

\begin{figure}[!htb]
\centering
\begin{tabular}{cc}
\includegraphics[width=0.44\textwidth]{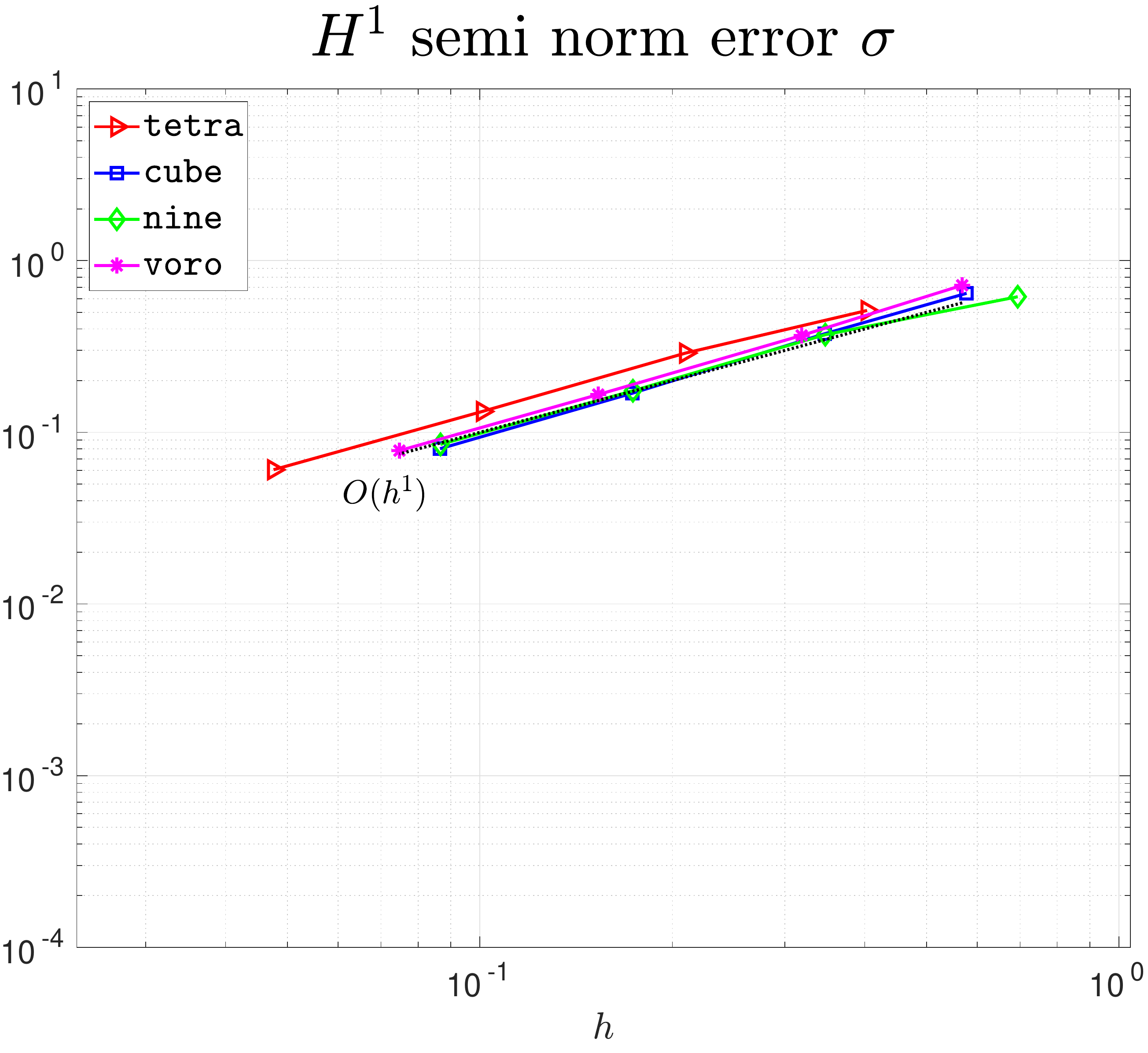}&
\includegraphics[width=0.44\textwidth]{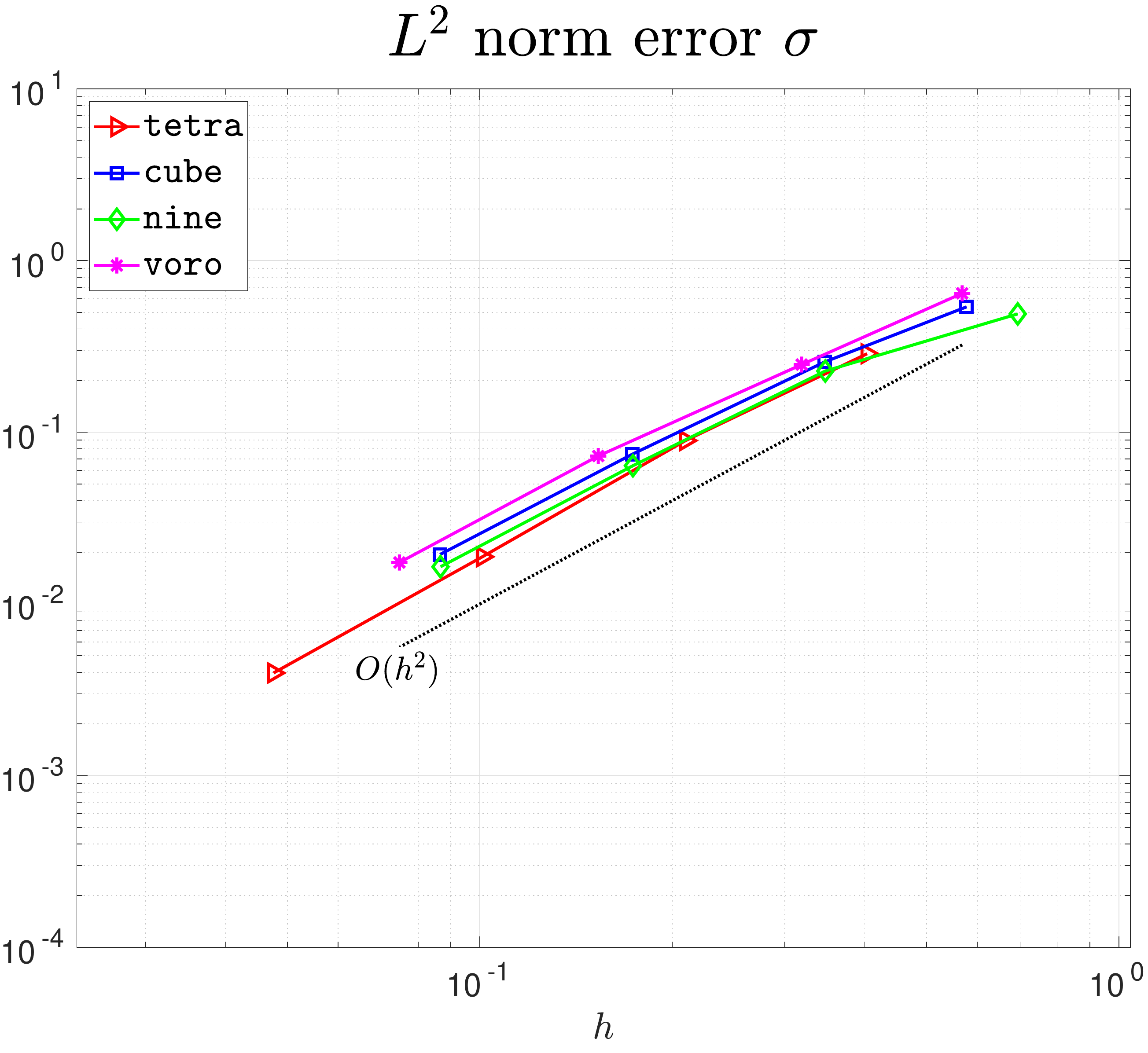}\\
\end{tabular}
\caption{Convergence Analysis: convergence rate of $\sigma$ each mesh type.}
\label{fig:exe2sigma}
\end{figure}

\bibliographystyle{abbrv} %siam, abbrv, ieeetr, unsrt, acm
\bibliography{references}
\end{document}